\documentclass[11pt,a4paper]{article}
\usepackage[T1]{fontenc}
\usepackage[utf8]{inputenc}
\usepackage{lmodern}
\usepackage{mathtools}
\usepackage{amsfonts}
\usepackage{amssymb}
\usepackage{graphicx}
\usepackage{amsthm}
\usepackage{verbatim}
\usepackage{hyperref}
\hypersetup{
    colorlinks=true,
    linkcolor=blue,
    filecolor=magenta,      
    urlcolor=cyan,
    citecolor=cyan 
}
\usepackage{cleveref}
\numberwithin{equation}{section}

\usepackage[showframe=false]{geometry}          
\usepackage{changepage}
\usepackage{authblk}
\usepackage{empheq}

\binoppenalty=\maxdimen
\relpenalty=\maxdimen

\theoremstyle{plain}
\newtheorem{theorem}{Theorem}[section]

\newtheorem{lemma}{Lemma}[section]

\theoremstyle{definition}

\newtheorem{definition}{Definition}[section]

\DeclareMathOperator{\Prb}{\mathbb{P}}

\DeclareMathOperator{\Expt}{\mathbb{E}}

\newcommand{\Ind}{\mathbf{1}}
\DeclareMathOperator{\Ext}{\mathsf{Ext}}

\newcommand{\Rd}{\mathbb{R}^d}
\newcommand{\Zd}{\mathbb{Z}^d}

\newcommand{\GO}{\mathcal{O}}
\newcommand{\F}{\mathcal{F}}

\newcommand{\X}{\mathcal{X}}
\newcommand{\Y}{\mathcal{Y}}
\newcommand{\Z}{\mathcal{Z}}

\newcommand{\R}{\mathcal{R}}
\newcommand{\cu}{\square}

\newcommand{\ub}{\bar{u}}
\newcommand{\ut}{\widetilde{u}}

\newcommand{\vb}{\bar{v}}
\newcommand{\Ur}{U_r}
\newcommand{\aL}{\underline{L}}
\newcommand{\aH}{\underline{H}}
\newcommand{\ec}{\lambda}
\newcommand{\phim}{\phi^{(\ec)}}
\newcommand{\Sm}{\S^{(\ec)}}
\newcommand{\ab}{\bar{\a}}
\newcommand{\A}{\mathbf{A}}
\newcommand{\f}{\mathbf{f}}

\newcommand{\p}[1]{\partial_{x_{#1}}}

\newcommand{\Fi}{\mathbf{F}}
\newcommand{\supp}{\text{supp}}

\renewcommand{\le}{\leqslant}
\renewcommand{\ge}{\geqslant}
\renewcommand{\leq}{\leqslant}
\renewcommand{\geq}{\geqslant}
\renewcommand{\epsilon}{\varepsilon}
\renewcommand{\S}{\mathbf{S}}
\renewcommand{\a}{\mathbf{a}}
\renewcommand{\u}{\mathbf{u}}

\newcommand*\widefbox[1]{\fbox{\hspace{1em}#1\hspace{1em}}}

\theoremstyle{remark}

\newtheorem*{remark}{Remark}

\author[1,2]{Chenlin GU}
\affil[1]{Ecole Polytechnique, Palaiseau, France}
\affil[2]{DMA, Ecole Normale Supérieure, PSL Research University, Paris, France}
\title{Uniform estimate of an iterative method for elliptic problems with rapidly oscillating coefficients}

\begin{document}
\maketitle

\begin{abstract}
We study the iterative algorithm proposed by S.\ Armstrong, A.\ Hannukainen, T.\ Kuusi, J.-C.\ Mourrat in \cite{armstrong2018iterative} to solve elliptic equations in divergence form with stochastic stationary coefficients. Such equations display rapidly oscillating coefficients and thus usually require very expensive numerical calculations, while this iterative method is comparatively easy to compute. In this article, we strengthen the estimate for the contraction factor achieved by one iteration of the algorithm. We obtain an estimate that holds uniformly over the initial function in the iteration, and which grows only logarithmically with the size of the domain. 
\end{abstract}

\section{Introduction}
\subsection{Main theorem}
The problem of homogenization is a subject widely studied in mathematics and other disciplines for its applications and interesting properties. Let $(\a(x),x \in \Rd)$ be a random coefficient field, which takes values in the set of $\mathbb{R}^{d \times d}$ symmetric matrices, and which we assume to be $\Zd$-stationary, with a unit range of dependence and uniformly elliptic that $ \Lambda^{-1} \vert \xi \vert^2 \leq \xi \cdot \a(x) \xi \leq \Lambda \vert \xi \vert^2$ for any $x, \xi \in \Rd$. We give ourselves a bounded domain $U \subset \Rd$ with boundary $C^{1,1}$, a scale parameter $0 < \epsilon < 1$, and for given $f \in H^{-1}(U)$ and $g \in H^1(U)$, we consider the elliptic Dirichlet problem 
\begin{equation}
\label{eq:mainEpsilon}
\left\{
	\begin{array}{ll}
	-\nabla \cdot \left(\a\left(\frac{\cdot}{\epsilon}\right) \nabla u_{\epsilon} \right) = f  & \text{ in } U, \\
	u_{\epsilon} = g & \text{ on } \partial U.
	\end{array}
\right.
\end{equation}
For the scale $0 < \epsilon \ll 1$, a naive numerical algorithm for this problem is generally very expensive, due to the rapid oscillations of the coefficients (comparatively to the size of the domain) and we have to refine the mesh of the numerical schema. Thus, different methods have been proposed to approximate the solution and one of them is to replace the conductance matrix $\a$ by a constant \textit{effective conductance matrix} $\ab$ in \cref{eq:mainEpsilon} and use its solution $\ub$ as an approximation, which can be solved quickly thanks to the multi-grid algorithm. However, $\ub$ is close to $u_{\epsilon}$ in the sense $L^2(U)$ or $H^{-1}(U)$, but not in some stronger topology, for example $H^1(U)$. Furthermore, the approximation only becomes accurate in the limit $\epsilon \rightarrow 0$, but for a small finite scale $\epsilon$, one can not expect a precision much smaller than $\epsilon$ with $\ub$.

Recently, \cite{armstrong2018iterative} proposed an iterative algorithm to solve the problem \cref{eq:main} efficiently for a given $\epsilon$-scale and with a better precision. We recap at first their algorithm here with the same formulation in large scale: Instead of considering \cref{eq:mainEpsilon} with a small scale $\epsilon$, we treat the Dirichlet problem with a dilation parameter $r \ge 1$, and set $\Ur := r U$. Then given $f \in H^{-1}(\Ur)$ and $g \in H^1(U_r)$, we consider the elliptic equation given by
\begin{equation}
\label{eq:main}
\left\{
	\begin{array}{ll}
	-\nabla \cdot \a \nabla u = f  & \text{ in } \Ur, \\
	u = g & \text{ on } \partial \Ur,
	\end{array}
\right.
\end{equation} 
and \cite{armstrong2018iterative} proposes to start by an initial guess of solution $v \in g +  H_0^1(U_r)$, and solve $u_0, \ub, \ut \in H^1_0(U)$ satisfying (with null Dirichlet boundary condition)
\begin{equation}
\label{eq:iterative}
\left\{
	\begin{array}{lll}
	(\lambda^2 - \nabla \cdot \a \nabla)u_0 &= f + \nabla \cdot \a \nabla v  & \text{ in } \Ur,\\
	-\nabla \cdot \ab \nabla \ub &= \lambda^2 u_0  & \text{ in } \Ur,\\
	(\lambda^2 - \nabla \cdot \a \nabla) \ut &= (\lambda^2 - \nabla \cdot \ab \nabla) \ub & \text{ in } \Ur.\\
	\end{array}
\right.
\end{equation}
Then, the iteration $\hat{v} := v + u_0 + \tilde{u}$ is a contraction. Since this rate of contraction is random, to estimate its size, \cite{armstrong2018iterative} introduces the notation $\GO_s$ for random variable $X$ that for any $s,\theta > 0$,
\begin{equation}\label{eq:defOinform}
X \leq \GO_s(\theta) \Longleftrightarrow \Expt\left[\exp(\max(\theta^{-1}X, 0)^s)\right] \leq 2.
\end{equation}
Informally speaking, the statement $X \leq \GO_s(1)$ tells us that $X$ has a tail lighter than $\exp(-x^s)$. 
We also introduce the shorthand notation for every $\lambda \in (0, 1]$, 
\begin{equation}
\ell(\lambda) := \left\{
	\begin{array}{ll}
	(\log (1+\lambda^{-1}))^{\frac{1}{2}} & d = 2, \\	
	1 & d > 2.
	\end{array}
\right.
\end{equation}
Then, \cite[Thoerem 1.1]{armstrong2018iterative} states that under a supplementary condition that the coefficient field $(\a(x), x \in \Rd)$ is $\alpha$-H\"older, for any $s \in (0, 2)$, we have a positive finite constant $C(s, U, \Lambda, \alpha, d)$ such that for the algorithm \cref{eq:iterative} in a domain $\Ur$ with $r \geq 1$, we have
\begin{equation}\label{eq:Old}
\Vert \nabla(\hat{v} - u) \Vert_{L^2(\Ur)} \leq \GO_s\left(C\ell^{\frac{1}{2}}(\lambda) \lambda^{\frac{1}{2}} \Vert \nabla(v - u) \Vert_{L^2(\Ur)}\right).
\end{equation}
However, in \cref{eq:Old} the contraction factor is proved for a given initialisation, but cannot be iterated to guarantee the convergence of the whole procedure. More precisely, after one iteration, the initial data becomes random and then we cannot make use of the noation $\GO_s$ \cref{eq:defOinform} with a random $\theta$ directly. In order to go past this obstacle, the authors of \cite{armstrong2018iterative} mention the possibility to get a uniform bound in \cite[eq.(1.10)]{armstrong2018iterative}, as is necessary to guarantee the convergence of the iterated method. The goal of this article is to confirm this idea and provide a proof of this uniform bound.

\smallskip

Here is our main theorem. The assumption about the random field $(\a(x), x \in \Rd)$ is the one given at the beginning of the introduction, and its precise definition can be found in \Cref{subsec:Field}, where $\Lambda$ is the constant of the uniform ellipticity condition. 
\begin{theorem}[Uniform $H^1$ contraction]
\label{thm:main}
For every bounded domain $U \subset \Rd$ with $C^{1,1}$ boundary and every $s \in (0,2)$, there exists a positive finite constant $C(U,\Lambda,s,d)$ and, for every $r \ge 2$ and $\lambda \in \left(\frac{1}{r}, \frac{1}{2} \right)$, a random variable $\Z$ satisfying
\begin{equation}
\label{e.estim.Z}
\Z \leq \GO_s\left(C \ell(\lambda)^{\frac{1}{2}}\lambda^{\frac{1}{2}}(\log r)^{\frac{1}{s}}\right),
\end{equation} 
such that the following holds.
Denote $\Ur := rU$, let $f\in H^{-1}(\Ur)$, $g\in H^1(\Ur)$, $v \in g + H^1_0(\Ur)$, let $u \in g + H^1_0(\Ur)$ be the solution of \cref{eq:main}, and let $u_0, \ub, \ut \in H^1_0(U)$ solve \cref{eq:iterative} with null Dirichlet boundary condition. Then for $\hat{v}:= v + u_0 + \ut$, we have the contraction estimate
\begin{equation}
\label{eq:contraction}
\Vert \nabla(\hat{v} - u) \Vert_{L^2(\Ur)} \leq  \Z \Vert \nabla(v - u) \Vert_{L^2(\Ur)}.
\end{equation}
\end{theorem}  
Compared with the result of \cite{armstrong2018iterative}, we have two fundamental improvements: The first one is the explicit random variable $\Z$, which is an upper bound for the contraction factor in the iteration, \emph{does not depend on the function} $v$. Hence, the algorithm can be iterated, replacing $v$ by $\hat v$, etc. In the event that $\Z \le \frac 1 2$, say, we thus obtain exponential convergence of the iterative method to the solution $u$, a conclusion that cannot be inferred from the results of \cite{armstrong2018iterative}. By the estimate \cref{e.estim.Z}, in order to guarantee that $\Z \le \frac 1 2$ with high probability, it suffices to take $\lambda$ sufficiently small that $\ell(\lambda)^{\frac{1}{2}}\lambda^{\frac{1}{2}}(\log r)^{\frac{1}{s}}$ is below a certain positive constant. The second improvement is that we do not make any assumption on the regularity of the coefficient field $x \mapsto \a(x)$, while \cite{armstrong2018iterative} assumed it to be H\"older continuous. 

The price is that the bound \cref{e.estim.Z} has another factor $(\log r)^{\frac{1}{s}}$ than \cref{eq:Old} and this point is also conjectured in \cite[eq.(1.10)]{armstrong2018iterative}. This factor does not weaken our algorithm viewing the range of $\lambda$ and more detailed study will be discussed in \Cref{subsec:Heuristic}. In \cite[Section 4]{armstrong2018iterative}, one can find some examples for the practical choice of $\lambda$ and numerical experiences. The author has also repeated the algorithm in a domain $128 \times 128$ with $\lambda = 0.1$, for $\a$ of type chessboard with law Bernoulli$(\frac{1}{2})$ taking value in $\left\{\frac{1}{\sqrt{2}}, \sqrt{2}\right\}$, and it takes several seconds to obtain a precision $10^{-5}$ on a laptop. In \cite{gu2019efficient}, this algorithm is applied to the \emph{supercritical percolation} setting and one can find the numerical results in Section 6. Hopefully, this algorithm also works on other stochastic homogenization models with stationary ergodic random coefficient field.

\begin{remark}
One can also state the algorithm \cref{eq:iterative} and \Cref{thm:main} in $\epsilon$-scale for \cref{eq:mainEpsilon}: In fact, by a simple change of variable that $\epsilon = \frac{1}{r}$, and $u_{\epsilon}(\cdot) = u(\frac{\cdot}{\epsilon})$ for $u$ in \cref{eq:main}, then the same estimates \cref{eq:contraction} and \cref{e.estim.Z} hold in the domain $U$ for \cref{eq:mainEpsilon} with $0 < \epsilon \leq \frac{1}{2}$, when applying the following iteration with $\ec \in \left(\epsilon, 1\right)$ 
\begin{equation}
\label{eq:iterativeEpsilon}
\left\{
	\begin{array}{lll}
	\left(\left(\frac{\ec}{\epsilon}\right)^2 - \nabla \cdot \left(\a\left(\frac{\cdot}{\epsilon}\right) \nabla\right)\right)u_0 &= f + \nabla \cdot \left(\a\left(\frac{\cdot}{\epsilon}\right) \nabla\right) v  & \text{ in } U,\\
	-\nabla \cdot \ab \nabla \ub &= \left(\frac{\ec}{\epsilon}\right)^2 u_0  & \text{ in } U,\\
	\left(\left(\frac{\ec}{\epsilon}\right)^2 - \nabla \cdot \left(\a\left(\frac{\cdot}{\epsilon}\right) \nabla\right)\right) \ut &= \left(\left(\frac{\ec}{\epsilon}\right)^2 - \nabla \cdot \ab \nabla\right) \ub & \text{ in } U.\\
	\end{array}
\right.
\end{equation}
\end{remark}

\smallskip

There exists a large amount of references about the homogenization theory and how we apply them in numerical solution. For example, see \cite{bensoussan2011asymptotic,kozlov1979averaging, tartar2009general,jikov2012homogenization,
allaire1992homogenization, yurinskii1986averaging, naddaf1998estimates} for the classical homogenization theory and see \cite{BL,EGH,GGS,OZB,malpet,KY,roulette, knap1,knap2,EL1,EL2, HMM} for the multi-grid algorithm in homogenization problem. To analyze a numerical algorithm for stochastic homogenization problem, it requires quantitative description and it was open for long time.
Thanks to the recent progress in a series of works of Armstrong, Kuusi, Mourrat and Smart \cite{armstrong2016lipschitz, armstrong2016mesoscopic, armstrong2016quantitative, armstrong2017additive}, and also the works of Gloria, Neukamm and Otto \cite{gloria2011optimal,gloria2012optimal,
gloria2014optimal,gloria2014regularity,
gloria2015quantification, gloria2015corrector}, we get a further understanding in this direction; see also the \cite{armstrong2018quantitative} a monograph and \cite{JCMIntro} as a brief introduction. In both this work and \cite{armstrong2018iterative}, the analysis depends on two-scale expansion theorem, which is introduced in \cite{allaire1992homogenization} in periodic case and \cite[Theorem 2.2, Theorem 2.3]{AllaireAmar} gives its rate of convergence; the quantitative analysis for this problem with random coefficient is studied in \cite{gloria2014optimal} and  \cite[Chapter 6]{armstrong2018quantitative}. Finally, we remark that in all the context, we suppose that the effective conductance $\ab$ is known, because this part is now well understood and there exist many efficient methods to do it quickly, see for example \cite{gloria2012numerical, egloffe2014random, mourrat2016efficient,fischer2018choice, hannukainen2019computing}.

\smallskip
The rest of the paper is organised as follows: At the end of the introduction, we focus on the numerical part to study why this algorithm is more efficient in \Cref{subsec:Complex}, and explain heuristically how this algorithm converges to the solution in \Cref{subsec:Heuristic}. \Cref{sec:notation} introduces some notations and then we turn to the proof of \Cref{thm:main}. The main improvements compared to \cite{armstrong2018quantitative} are the two technical lemmas in \Cref{sec:improve}, then we put our technique in the proof of \cite{armstrong2018iterative}, which is a quantitative two-scale expansion theorem, and we reformulate  it in \Cref{sec:two-scale}. Finally, in \Cref{sec:iteration}, we combine all the results and obtain the main theorem.

\subsection{Complexity analysis in numeric}\label{subsec:Complex}
In this part we give a numerical consequence of \Cref{thm:main}. We start by recalling why solving for $(\lambda^2 - \nabla \cdot \a \nabla)$ is computationally less difficult than solving for $- \nabla \cdot \a \nabla$. In our context, after discretization, the elliptic equations is transformed to a symmetric linear system 
\begin{equation}\label{eq:System}
\A \u = \f,
\end{equation}
where $\A \in \mathbb{R}^{N \times N}$ is positive definite, $\u,\f \in \mathbb{R}^N$ and $N$ stands for the number of elements which is fixed during all this subsection. To capture all the information of coefficients, the minimal numerical resolution requires that $N = O(r^d)$. Then the problem becomes solving a large sparse linear system. 

One basic method for this problem is the \textit{conjugate gradient method} (CGM), whose rate of convergence is 
$$
\tau(\A) = \frac{\sqrt{\rho(\A)} -1}{\sqrt{\rho(\A)} + 1},
$$
where $\rho$ is the \textit{spectral condition number}  defined as
$$
\rho(\A) = \frac{\kappa_{max}(\A)}{\kappa_{min}(\A)},
$$
and $\kappa_{max}, \kappa_{min}$ stands for the maximum and minimum eigenvalues ( \cite[Theorem 6.29, eq.(6.128)]{sadd2003}). In practice, $\kappa_{max}(\A) \approx \text{constant}$ while $\kappa_{min}(\A) \approx r^{-2}$. Thus, when $r$ grows bigger, the ratio of convergence $\tau(\A) \approx 1-\frac{1}{r}$. It is still a geometric convergence but the rate is very small and to solve \cref{eq:System} with a resolution $\epsilon_0$, it requires $O\left( r \vert \log(\epsilon_0)\vert \right)$ rounds of CGM.

Now we focus on the complexity to solve \cref{eq:System} with \cref{eq:iterative}. Since in every iteration we solve two regularised equations and a homogenized equation, we investigate their complexity at first:
\begin{itemize}
\item  For the homogenized equation, since the matrix is constant, which alows us to apply the multi-grid algorithm and for a resolution $\epsilon_1$, the complexity is $O(\vert \log(\epsilon_1)\vert)$ rounds CGM \cite[Chapter 4]{multgridtutor}. 
\item For the regularised equation $(\lambda^2 \mathsf{Id} + \A)\u^{\lambda} = \f$, we use CGM and the spectral condition number for $\frac{1}{r} \ll \lambda \ll 1$ is  
$$
\rho(\lambda^2 \mathsf{Id} + \A) = \frac{\lambda^2 + \kappa_{max}(\A)}{\lambda^2 + \kappa_{min}(\A)} \approx \frac{C}{\lambda^2},
$$
and this operation also changes the typical size of the rate of convergence 
$$
\tau(\lambda^2\mathsf{Id} + \A) = \frac{\sqrt{\rho(\lambda^2\mathsf{Id} + \A)} -1}{\sqrt{\rho(\lambda^2\mathsf{Id} + \A)} + 1} \approx 1 - \frac{\lambda}{C}.
$$
Then for a resolution $\epsilon_1$, it requires $O(\lambda^{-1}\vert \log(\epsilon_1)\vert)$.
\end{itemize} 
When we implement the algorithm \cref{eq:iterative}, generally speaking, the \cref{e.estim.Z} tells us with a large probability, after every iteration the precision will be multiplied $\lambda^{\frac{1}{2}}$. Thus, totally it demands $O\left(\vert \log(\lambda) \vert^{-1} \vert \log(\epsilon_0) \vert\right)$ iterations. Moreover, in the k-th iteration, it suffices to obtain a resolution $\epsilon_1 = \lambda^{\frac{k}{2}}$ for the regularised equation and the homogenized equation, so the complexity is 
$$
\sum_{k=1}^{O\left(\vert \log(\lambda) \vert^{-1} \vert \log(\epsilon_0) \vert\right)} \lambda^{-1}\vert \log(\lambda^{\frac{k}{2}})\vert \approx \vert \log(\lambda) \vert^{-1} \lambda^{-1}\log^{2}(\epsilon_0).
$$
When we take a typical choice of $\lambda = (\log r)^{-1}$, the complexity of the iterative algorithm is $O(\log r \vert \log \epsilon_0 \vert^2)$ rounds of CGM. This quanity is much smaller than the complexity of solving \cref{eq:System} directly for a big $r$, provided that the precision is reasonable, for example $\epsilon_0 = r^{-n}$ for $n$ fixed.

\subsection{Heuristic analysis of algorithm}
\label{subsec:Heuristic}
\begin{figure}[!h]
\centering
\includegraphics[scale=0.60]{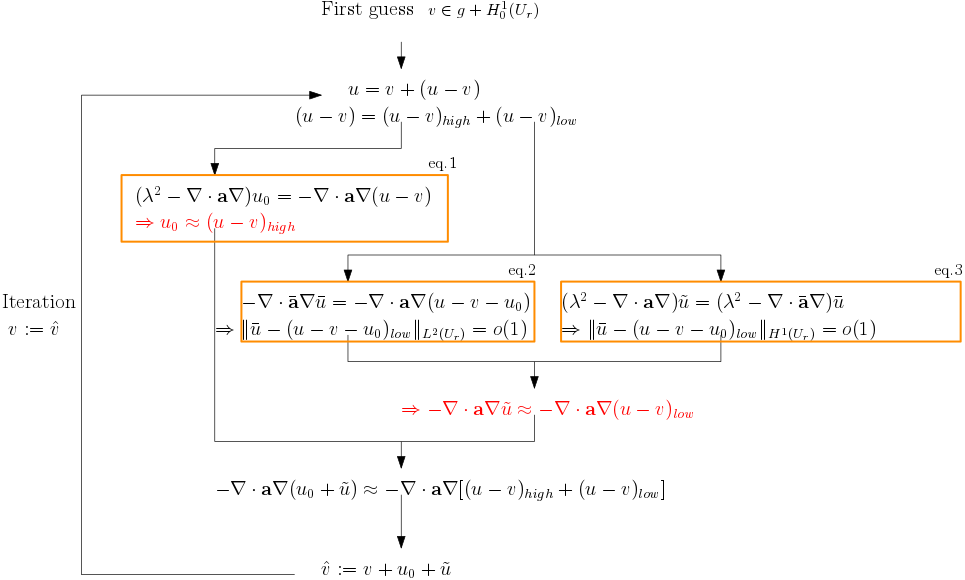}
\caption{A flowchart shows the mechanic of the algorithm.}
\end{figure}
In this part, we explain heuristicly why the iterative algorthm converges to the true solution and how we figure out this algorithm. We keep in mind the main idea: To solve \cref{eq:main}, we divide it into sub-problems of regularized equations and homogenized equations.

We start with an initial guess of the solution $v$, then we write $u = v + (u-v)$ and we want to recover the part $(u-v)$. Since the divergence form is linear, we have
$$
-\nabla \cdot \a \nabla (u-v) = -\nabla \cdot \a \nabla u + \nabla \cdot \a \nabla v = f +\nabla \cdot \a \nabla v.
$$
In the first step of our algorithm, we solve the problem with regularization 
$$
(\lambda^2 -\nabla \cdot \a \nabla) u_0 = f +\nabla \cdot \a \nabla v,
$$
and $u_0$ gives the high frequency part of $(u-v)$ associated to the operator $-\nabla \cdot \a \nabla$ on $\Ur$. To see it, we apply the theorem of spectral decomposition \cite[Chapter 5]{simon1980methods} 
$$
(u - v) = \sum_{i}^{\infty} \psi_i,
$$
where $\psi_i$ is the projection on the subspace of the eigenvalue $\kappa_i$ associated to $-\nabla \cdot \a \nabla$ on $\Ur$,  with $0 < \kappa_1 < \kappa_2 \cdots$. Then $u_0$ has an expression
$$
u_0 = \sum_{i=1}^{\infty} \frac{\kappa_i}{\lambda^2 + \kappa_i} \psi_i.
$$ 
We see the projections associated to large eigenvalues have a small perturbation after the regularization. Thus, we consider the solution $u_0$ of the high frequency projection of $(u - v)$. Informally, we write
\begin{eqnarray*}
(u-v) &\approx & (u-v)_{high} + (u-v)_{low}, \\
u_0 &=& (u-v)_{high}.
\end{eqnarray*}
Therefore, after the first step, we do not get all the information of $(u-v)$ but $(u-v)_{high}$ and the second and the third equation serve to recover $(u-v)_{low}$ of $(u-v)$. Using the superposition, a direct idea is to solve 
\begin{equation}\label{eq:pseudo1}
-\nabla \cdot \a \nabla \ut \approx -\nabla \cdot \a \nabla (u-v-u_0) = \lambda^2 u_0.  
\end{equation}
But for the reason of less numerical cost, we choose to solve at first a homogenized problem
\begin{equation}
\label{eq:Rigorous1}
-\nabla \cdot \ab \nabla \ub = -\nabla \cdot \a \nabla (u-v-u_0) = \lambda^2 u_0,
\end{equation}
and if we believe that $-\nabla \cdot \a \nabla \ut \approx -\nabla \cdot \ab \nabla \ub $, we can also solve the one by adding regularization 
\begin{equation}
\label{eq:Rigorous2}
(\lambda^2 - \nabla \cdot \a \nabla) \ut  = (\lambda^2 - \nabla \cdot \ab \nabla) \ub. 
\end{equation}
Then, we hope that this $\ut$ gives us $(u-v)_{low}$. Perhaps this "$\approx$" is not very precise, so we do $\hat{v} := v + u_0 + \ut$ and put $\hat{v}$ in the role of $v$ for several iterations in order to get a further approach to the solution of \cref{eq:main}.

\section{Notation}\label{sec:notation}
In this section, we state our assumptions about the coefficient field precisely and introduce some notations.
\subsection{Assumptions on the coefficient field}\label{subsec:Field}
We denote by $\left((\a(x))_{x \in \Rd}, \F, \Prb\right)$ the probability space, and by $\F_{V}$ the $\sigma$-algebra generated by 
$$
\a \mapsto \int_{\Rd} \chi \a_{i,j}, \text{ where } i,j \in \{1,2,3 \cdots d\}, \chi \in C^{\infty}_c(V).
$$
$\F$ is short for $\F_{\Rd}$. $T_y$ denotes the operator of translation i.e. for any function $f$, $T_y(f)(x) := f(x+y)$  and for any set $A$, $T_y(A) := \{x+y \vert x \in A\}$.

The precise assumptions for the coefficient field are as follows.
\begin{enumerate}
\item $\Zd$-\textit{stationarity}: For each $A \in \F$ and each $y \in \Zd$, we have $\Prb[T_y(A)] = \Prb[A]$.
\item \textit{Unit range correlation}: 
$$\forall W, V  \in \mathcal{B}(\Rd),\quad  d_{H}(W,V) > 1 \implies \F_W, \F_V \text{ are independent. }$$ 
Here $d_H$ is the Hausdorff distance in $\Rd$.
\item \textit{Uniform ellipticity}: 
There exists $0 < \Lambda < \infty$ such that with probability one and for every $x,\xi \in \Rd$, we have
$\Lambda^{-1}|\xi|^2 \leq \xi \cdot \a(x) \xi \leq \Lambda |\xi|^2$.
\end{enumerate}

\subsection{Notation $\GO_s$}
We recall the definition of $\GO_s$
\begin{equation}
X \leq \GO_s(\theta) \iff \Expt\left[\exp((\theta^{-1}X)_{+}^s)\right] \leq 2,
\end{equation} 
where $(\theta^{-1}X)_+$ means $\max\{\theta^{-1}X, 0\}$. This notation gives the tail estimate of random variables: One can use Markov's inequality to obtain that 
$$
X \leq \GO_s(\theta) \Longrightarrow \forall x > 0, \Prb[X \geq \theta x] \leq 2 \exp(-x^s).
$$
Moreover, we can also obtain the estimate of the sum of a series of random variables without its joint distribution: For a measure space $(E, \mathcal{S}, m)$ and $\{X(z)\}_{z \in E}$ a family of random variables, we have 
\begin{equation}
\label{eq:OSum}
\forall z \in \GO_s(E), X(z )\leq \GO_s(\theta(z)) \Longrightarrow \int_{E} X(z) m(dz) \leq \GO_s\left( C_s \int_{E} \theta(z) m(dz) \right),
\end{equation}
where $0 < C_s < \infty$ is a constant and $C_s = 1$ for $s \geq 1$. See Appendix of \cite[Appendix A]{armstrong2018quantitative} for proofs and other operations on $\GO_s$.

\subsection{Convolution}
For $f \in L^p(\Rd), g \in L^q(\Rd)$ where $\frac{1}{p} + \frac{1}{q} = 1$, we denote by $f \star g$ the convolution of the function $f, g$ 
$$f \star g (x) = \int_{\Rd} f(y)g(x-y) \, dy.$$
In this article, two mollifiers used are the heat kernel $\Phi_r(x)$, defined for $r > 0$ and $x \in \Rd$ by
$$
\Phi_r(x) := \frac{1}{(4\pi r^2)^{d/2}}\exp\left(-\frac{x^2}{4r^2}\right),
$$
and the bump function $\zeta \in C^{\infty}_c(\Rd)$
$$
\zeta(x) := c_d \exp\left(-\left(\frac{1}{4}-|x|^2\right)^{-1}\right) \Ind_{\{|x|<\frac{1}{2}\}},
$$
where $c_d$ is the constant of normalization such that $\int_{\Rd} \zeta(x) dx = 1$. Finally, we use the notation 
$$
\zeta_{\epsilon}(x) = \frac{1}{\epsilon^{d}}\zeta\left( \frac{x}{\epsilon}\right),
$$
as a mollifier in scale $\epsilon > 0$ and we have $\supp(\zeta_{\epsilon}) \subset \overline{B_{\epsilon/2}}$.

\subsection{Function spaces}
We use $\{e_1, e_2, \cdots e_d\}$ as the canonical basis of $\Rd$. For every Borel set $U \subset \Rd$, let $\vert U \vert$ be its Lebesgue measure. For each $p \in [1, +\infty]$, we denote by $L^p(U)$ the classical Lebesgue space and $L^{p}_{loc}(\Rd)$ the function space for the functions with finite $L^{p}(V)$ norm for any compact set $V$. The weighted norm $\aL^p(U)$ is defined for a bounded Borel set $U$ as
$$
\Vert f \Vert_{\aL^p(U)} = \left( \frac{1}{|U|} \int_U |f(x)|^p  \, dx \right)^{\frac{1}{p}} = |U|^{-\frac{1}{p}}\Vert f \Vert_{L^p(U)}.
$$
For each $k \in \mathbb{N}$, we denote by $H^k(U)$ the classical Sobolev space on $U$ equipped with the norm
$$
\Vert f \Vert_{H^k(U)} := \sum_{0 \leq |\beta| \leq k} \Vert \partial^{\beta} f \Vert_{L^2(U)},
$$
where $\beta \in \mathbb{N}^d$ represents a multi-index weak derivative, 
$$
|\beta| := \sum_{i=1}^d \beta_i \quad \text{  and  } \quad \partial^{\beta} f = \p{1}^{\beta_1} \cdots \p{d}^{\beta_d}f.
$$
We also use $\vert\nabla^k f \vert$ to indicate $\sum_{|\beta| = k} \vert\partial^{\beta} f\vert$. When $|U| < \infty$, we define the weighted norm that 
$$
\Vert f \Vert_{\aH^k(U)} := \sum_{0 \leq |\beta| \leq k} |U|^{\frac{|\beta|-k}{d}} \Vert \partial^{\beta} f \Vert_{\aL^2(U)}.
$$
$H^k_0(U)$ denotes the closure of $C_c^{\infty}(U)$ in $H^k(U)$ and $H^{-1}(U)$ for the dual of $H^1(U)$. The weighted norm $\aH^{-1}(U)$ is
$$
\Vert f \Vert_{\aH^{-1}(U)} := \sup \left\{ \frac{1}{|U|}\int_{U} f(x)g(x) \,dx , g \in H^{1}_0(U), \Vert g \Vert_{\aH^1(U)} \leq 1 \right\}.
$$
Here, we abuse the use of the integration $\int_{U} f(x)g(x) \,dx $, since the function space $H^{-1}(U)$ also contains linear functional, which is not necessarily a function. 

Finally, we remark that one advantage of the definition of $\aH^k$ is that it is consistent with the scale constant of Poincar\'e's inequality \cite[eq.(7.44)]{trudinger} and Sobolev extension theorem \cite[Theorem 7.25]{trudinger}. That is, under the condition that the Borel set $U$ has $C^{1,1}$ boundary, for any function $f \in H^1_0(\Ur)$ 
\begin{equation}\label{eq:Poincare}
\Vert f \Vert_{\aH^1(\Ur)} \leq C(U,d)\Vert \nabla f \Vert_{\aL^2(\Ur)},
\end{equation}
and for any $f \in H^2(\Ur)$, there exists an extension $\Ext(f) \in H_0^2(\Rd)$ such that $\Ext(f) \equiv f$ in $\Ur$
\begin{equation}\label{eq:Extension}
\sum_{0 \leq |\beta| \leq 2} |\Ur|^{- \frac{d}{2} - \frac{|\beta|-2}{d}} \Vert \partial^{\beta} \Ext(f) \Vert_{L^2(\Rd)} \leq C(U,d) \Vert f \Vert_{\aH^2(\Ur)}.
\end{equation}
The proof depends on the scaling argument: For \cref{eq:Poincare}, we prove at first the result in domain $U$ and then apply to $x \mapsto f(rx)$. For \cref{eq:Extension}, we apply \cite[Theorem 7.25]{trudinger} to the domain $U$ and obtain an extension $\Ext_U$ satisfying \cref{eq:Extension} for $r = 1$. Then, for the extension $\Ext_{\Ur}$ on the domain $\Ur$, we define $\Ext_{\Ur}(f)(x) := \Ext_{U}(f(r \cdot))(x/r)$ and this satisfies \cref{eq:Extension} with a constant depending only on $U, d$. In the following paragraphs we neglect the index of domain and still note the extension $\Ext$.

\subsection{Cubes}
We use the notation $\cu$ to refer the open unit cube 
$\cu := \left(-\frac{1}{2}, \frac{1}{2} \right)^d$. For any $z \in \Rd$, the translation of $\cu$ in the direction $z$ writes $z + \cu := z + \left(-\frac{1}{2}, \frac{1}{2} \right)^d$.
The sum of a cube and a Borel set $U$ is defined as 
$$
\cu + U := \left\{z \in \Rd | z = x + y, x \in \cu, y \in U \right\}.
$$
We also denote define scaling of the cube by size $\epsilon > 0$ that $\cu_{\epsilon} := \left(-\frac{\epsilon}{2}, \frac{\epsilon}{2} \right)^d.$

\section{Two technical lemmas}
\label{sec:improve}
We prove two useful lemmas that improves the estimate of the iterative algorithm in this section. A formula similar to \Cref{lem:MixedNorm} can be found in \cite[Lemme 6.7]{armstrong2018quantitative}. Here we introduce a variant version and it works well together with \Cref{lem:max}.

\subsection{An inequality of localization}
\begin{lemma}[Mixed norm]
\label{lem:MixedNorm}
There exists a constant $0 < C(d) < \infty$ such that for every $f \in L_{loc}^2(\Rd)$, $g \in L^2(\Rd)$ and every $\epsilon > 0, r \geq 2$, we have the following inequality
\begin{equation}\label{eq:MixedNorm}
\Vert f (g \star \zeta_{\epsilon})\Vert_{L^2(\Ur)} \leq C(d)\left( \max_{z \in \epsilon \Zd \cap (\Ur + \cu_{\epsilon})}\Vert f \Vert_{\aL^2(z + \cu_{\epsilon})} \right) \Vert g \Vert_{L^2(\Ur + \cu_{3\epsilon})}.
\end{equation}
\end{lemma}
\begin{proof}

\begin{figure}[h]
\centering
\includegraphics[scale=0.4]{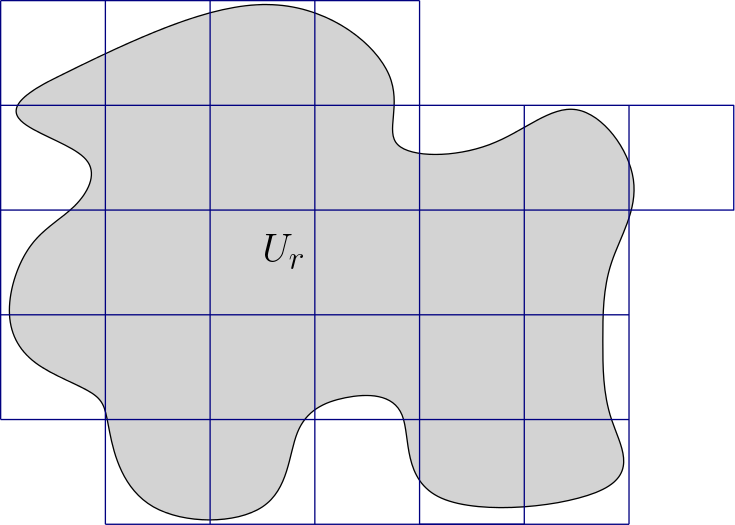}
\caption{We calculate the $L^2$ norm by the sum of all norm in small cubes of size $\epsilon$, so we counts all cubes in the domain $(\Ur + \cu_{\epsilon})$.}
\end{figure}
We decompose the $L^2$ norm as the sum of that in small cubes $\cu_{\epsilon}$
\begin{eqnarray*}
\Vert f (g \star \zeta_{\epsilon})\Vert^2_{L^2(\Ur)} &\leq & \sum_{z \in \epsilon \Zd \cap (\Ur + \cu_{\epsilon})} \Vert f (g \star \zeta_{\epsilon})\Vert^2_{L^2(z + \cu_{\epsilon})}  \\
(\text{H\"older's inequality}) &\leq & \sum_{z \in \epsilon \Zd \cap (\Ur + \cu_{\epsilon})} \left( \Vert f \Vert^2_{L^2(z + \cu_{\epsilon})}\Vert  g \star \zeta_{\epsilon}\Vert^2_{L^{\infty}(z + \cu_{\epsilon})} \right) \\
&\leq &  \left( \max_{z \in \epsilon \Zd \cap (\Ur+\cu_{\epsilon})}\Vert f \Vert^2_{L^2(z + \cu_{\epsilon})} \right) \left( \sum_{z \in \epsilon \Zd \cap (\Ur + \cu_{\epsilon})} \Vert  g \star \zeta_{\epsilon}\Vert^2_{L^{\infty}(z + \cu_{\epsilon})} \right).
\end{eqnarray*}
Noticing that for any $x \in z + \cu_{\epsilon}, y \in B_{\epsilon/2}$ then $(x - y) \in z + \cu_{2\epsilon}$ and we have
\begin{eqnarray*}
|g \star \zeta_{\epsilon}(x)| &=& \left| \int_{\cu_{\epsilon}} g(x -y) \frac{1}{\epsilon^d} \zeta(\frac{y}{\epsilon}) \, dy \right| \\
&\leq & \frac{C(d)}{\epsilon^d} \int_{z + \cu_{2\epsilon}} |g(y)| \, dy \\
&\leq & \frac{C(d)}{\epsilon^d} \left( \int_{z + \cu_{2\epsilon}} |g(y)|^2 \, dy \right)^{\frac{1}{2}} | \cu_{2\epsilon}|^{\frac{1}{2}}\\
&\leq & \frac{C(d)}{\epsilon^{\frac{d}{2}}}  \Vert g \Vert_{L^2(z + \cu_{2\epsilon})}. 
\end{eqnarray*}
So we get 
$$\Vert  g \star \zeta_{\epsilon}\Vert_{L^{\infty}(z + \cu_{\epsilon})} \leq  \frac{C(d)}{\epsilon^{\frac{d}{2}}} \Vert g \Vert_{L^2(z + \cu_{2\epsilon})},$$
and we add this analysis in the former inequality and obtain that 
\begin{eqnarray*}
\Vert f (g \star \zeta_{\epsilon})\Vert^2_{L^2(\Ur)} &\leq & C(d) \left( \frac{1}{\epsilon^d}\max_{z \in \epsilon \Zd \cap (\Ur + \cu_{\epsilon})}\Vert f \Vert^2_{L^2(z + \cu_{\epsilon})} \right) \left( \sum_{z \in \epsilon \Zd \cap (\Ur + \cu_{\epsilon})} \Vert g \Vert^2_{L^2(z + \cu_{2\epsilon})}\right) \\
&\leq & C(d)\left( \max_{z \in \epsilon \Zd \cap (\Ur + \cu_{\epsilon})}\Vert f \Vert^2_{\aL^2(z + \cu_{\epsilon})} \right) \Vert g \Vert^2_{L^2(\Ur+\cu_{3\epsilon})}.
\end{eqnarray*}
This is the desired inequality.
\end{proof}

\subsection{Maximum of finite number of random variables of type $\GO_s(1)$}
Since $\left( \max_{z \in \epsilon \Zd \cap (\Ur + \cu_{\epsilon})}\Vert f \Vert_{\aL^2(z + \cu_{\epsilon})} \right)$ often appears in the context as the maximum of a family of random variables, we prove the following lemma to analyze the maximum of a finite number of random variables of type $\GO_s(1)$. Note that we do not make any assumptions on the joint law of the random variables.

\begin{lemma}
\label{lem:max}
For all $N \geq 1$ and a family of random variables $\{X_i\}_{1\leq i \leq N}$ satisfying that $X_i \leq \GO_s(1)$, we have 
\begin{equation}
\label{eq:max}
\left( \max_{1 \leq i \leq N} X_i \right) \leq \GO_s\left( \left(\frac{\log(2N)}{\log(4/3)}\right)^{\frac{1}{s}}\right).
\end{equation}
\end{lemma}

\begin{proof}
For the case $N=1$, we could check that \cref{eq:max} is established since 
$$
M = X_1 \leq \GO_s(1) \Longrightarrow M \leq \GO_s\left(\frac{\log(4)}{\log(4/3)}\right),
$$
so we focus on the case $N \geq 2$. By Markov's inequality, $X_i \leq \GO_s(1)$ gives us $\Prb[X_i > x] \leq 2 e^{-x^s}.$ Then we use the union bound to get 
\begin{equation}\label{eq:BoundProb}
\Prb\left[\max_{1 \leq i \leq N} X_i > x\right] = \Prb\left[\bigcup_{i=1}^N \{X_i > x\}\right] \leq  \left( 1 \wedge \sum_{i=1}^N \Prb[X_i > x]\right) \leq 1 \wedge 2N e^{-x^s}.
\end{equation}
We denote by $x_0$ the critical point such that $e^{x_0^s} = 2N$ and $a > 0$ such that $a^s > 3$ and its value will be chosen carefully later. We also set $M = \max_{1 \leq i \leq N} X_i $ and use Fubini's theorem: For an increasing positive function $g \in C^{1}(\mathbb{R})$ such that $g(0) = 0$, we have 
\begin{equation}\label{eq:Fubini}
\Expt\left[g(M)\right] = \int_{0}^{\infty} g(t)\Prb[M > t] \, dt.
\end{equation}  
We apply \cref{eq:Fubini} the function $x \mapsto \exp\left(\left(\frac{x}{a}\right)_{+}^s\right) - 1$, 
\begin{equation}\label{eq:maxDecom}
\begin{split}
\Expt\left[ \exp\left(\left(\frac{M}{a}\right)_{+}^s\right)\right] &= \int_0^{\infty} \frac{s}{a} \left(\frac{x}{a}\right)^{s-1} e^{(\frac{x}{a})^s}\Prb[M > x] \, dx  + 1 \\
&= \int_0^{x_0} \frac{s}{a} \left(\frac{x}{a}\right)^{s-1} e^{(\frac{x}{a})^s} \Prb[M > x] \, dx \\ 
 & \quad +  \int_{x_0}^{\infty} \frac{s}{a} \left(\frac{x}{a}\right)^{s-1} e^{(\frac{x}{a})^s} \Prb[M > x] \, dx + 1.
\end{split}
\end{equation}
For the integration on the interval $(0, x_0]$, we use the estimate \cref{eq:BoundProb} that $\Prb[M > x] \leq 1$ and $e^{x_0^s} = 2N$ to bound it
\begin{equation}\label{eq:maxPart1}
\begin{split}
\int_0^{x_0} \frac{s}{a} \left(\frac{x}{a}\right)^{s-1} e^{(\frac{x}{a})^s} \Prb[M > x] \, dx &\leq   \int_0^{x_0} \frac{s}{a} \left(\frac{x}{a}\right)^{s-1} e^{(\frac{x}{a})^s} \, dx \\
&=  \int_0^{x_0}  e^{(\frac{x}{a})^s} \, d \left(\frac{x}{a}\right)^s \\
&= (2N)^{\frac{1}{a^s}} - 1. 
\end{split}
\end{equation}
Similarly, for the integration on the interval $(x_0, \infty)$, we use \cref{eq:BoundProb} to control the probability that $\Prb[M > x] \leq 2N e^{-x^s}$ and give an estimate
\begin{equation}\label{eq:maxPart2}
\begin{split}
\int_{x_0}^{\infty} \frac{s}{a} \left(\frac{x}{a}\right)^{s-1} e^{(\frac{x}{a})^s} \Prb[M > x] \, dx &\leq 
2N\int_{x_0}^{\infty} \frac{s}{a} \left(\frac{x}{a}\right)^{s-1} e^{(\frac{x}{a})^s-x^s} \, dx \\
&=  2N \int_{x_0}^{\infty}  e^{(\frac{x}{a})^s-x^s}  \, d \left(\frac{x}{a}\right)^s \\
&=  \underbrace{\frac{1}{a^s -1}}_{\text{Using }a^s \geq 3}(2N)^{\frac{1}{a^s}} \\
&\leq \frac{1}{2}(2N)^{\frac{1}{a^s}}.\\
\end{split}
\end{equation}
Now we fix $a = \left(\frac{\log(2N)}{\log(4/3)}\right)^{\frac{1}{s}}$. For the case $N \geq 2$, we can check that
$$
a^s = \left(\frac{\log(2N)}{\log(4/3)}\right) \geq \left(\frac{\log(4)}{\log(4/3)}\right) > 3,
$$ 
so $a^s \geq 3$ is satisfied. Finally, we put back the estimate \cref{eq:maxPart1} and \cref{eq:maxPart2} back to \cref{eq:maxDecom} and verify the definition of $\GO_s$
$$
\Expt\left[ \exp\left(\left(\frac{M}{a}\right)_{+}^s\right)\right] \leq \frac{3}{2}(2N)^{\frac{1}{a^s}} = \frac{3}{2} e^{\log(4/3)} = 2.
$$
This finishes the proof.
\end{proof}

\section{Two-scale expansion estimate}\label{sec:two-scale}
This section reformulates \cite[Theorem 3.1]{armstrong2018iterative} a quantitative two-scale expansion theorem with the improvements from the lemmas in \Cref{sec:improve}.
\subsection{Main structure}
The two-scale expansion allows us to approximate the solution of \cref{eq:main} by the solution of the homogenized equation and the \textit{first order corrector} $\{\phi_{e_k}\}_{1 \leq k \leq d}$. We recall at first the definition of the first order corrector.
\begin{definition}[First order corrector, Lemma 3.16 and    Theorem 4.1 of \cite{armstrong2018quantitative}]\label{def:Corrector}
For each $p \in \Rd$, the corrector $\phi_p$ is the sublinear function satisfying that $p \cdot x + \phi_p$ is $\a$-harmonic in whole space $\Rd$ i.e.
\begin{eqnarray}
- \nabla \cdot \a(p + \nabla \phi_p) = 0 , \text{ in } \Rd.
\end{eqnarray}
$\nabla \phi_p$ is $\Zd$-stationary and $\phi_p$ is well-defined up to a constant. For $d \geq 3$, we can choose a constant such that $\Expt\left[\int_{\cu} \phi_p \right] = 0$, and in this case $\phi_p$ is also $\Zd$-stationary.
\end{definition} 
 
We remark that proof of the property $\phi_{p}$ is $\Zd$-stationary for $d \geq 3$ requires a detailed quantitative study of the following modified corrector
\begin{align}\label{eq:CorrectorModified}
\phim_{p} := \phi_p - \phi_p \star \Phi_{\ec^{-1}},
\end{align}
which is $\Zd$-stationary and is always well-defined. Then the quantitative study of $\phim_{p}$ allows us to extract a limit in the subsequence when $\ec \rightarrow 0$ and this gives us a candidate for the choice of constant. For its complete proof, see \cite[Chapter 4.6, Page 175]{armstrong2018quantitative}.  
In the following paragraphs, we specify the constant by $\Expt\left[\int_{\cu} \phi_p \right] = 0$. Therefore, we can use the property that $\phi_{p}$ is $\Zd$-stationary for $d \geq 3$.

Once we define the first-order corrctor, as well-known from \cite{allaire1992homogenization}, for $v$ the heterogeneous and $\vb$ the homogenized solution, the two-scale expansion $\vb + \sum_{k=1}^d (\p{k} \vb)  \phi_{e_k}$ approximates $v$ in the sense $H^1$. We follow the similar idea in \cite[Chapter 6]{armstrong2018quantitative} and apply a modified two-scale expansion to $\vb \in H^1_0(\Ur) \cap H^2(\Ur)$ with $\{\phim_{e_k}\}_{1 \leq k \leq d}$ defind in \cref{eq:CorrectorModified}
\begin{align}\label{eq:defTwoScale}
w := \vb +  \sum_{k=1}^d \p{k} (\Ext(\vb) \star \zeta) \phim_{e_k},
\end{align}
where $\Ext(\vb)$ is the Sobolev extension such that \cref{eq:Extension} is established. In the following we abuse a little the notation to identify 
\begin{align}\label{eq:ExtCov}
\vb \star \zeta := \Ext(\vb) \star \zeta,
\end{align}
in order to shorter the equation. We want to prove a $H^1$ convergence theorem for the operator $(\mu^2 - \nabla \cdot \a \nabla)$. We also recall the definition of $\ell(\lambda)$
\begin{equation*}
\ell(\lambda) = \left\{
	\begin{array}{ll}
	(\log (1+\lambda^{-1}))^{\frac{1}{2}} & d = 2, \\	
	1 & d > 2.
	\end{array}
\right.
\end{equation*}

\begin{theorem}[Two-scale estimate]

\label{thm:TwoScale}
For every $r\geq 2, \lambda \in (\frac{1}{r}, \frac{1}{2})$, there exists three $\F$-measurable random variables $\X_1, \X_2, \Y_1$ satisfying for each $s \in (0,2)$, there exists a constant $C'(U,s, d)$ such that the following holds: $\X_1, \X_2, \Y_1$ have an estimate
\begin{eqnarray}\label{eq:tableGO}
\X_1 \leq \GO_s\left(C'(U,s,d) \ell(\ec)(\log r)^{\frac{1}{s}} \right), & & \X_2 \leq \GO_s\left(C'(U,s,d) \ec^{\frac{d}{2}}(\log r)^{\frac{1}{s}} \right), \\
\Y_1 \leq \GO_s\left(C'(U,s,d) \ell(\ec)(\log r)^{\frac{1}{s}} \right),
\end{eqnarray}
and for every $\vb \in H^1_0(\Ur) \cap H^2(\Ur), \mu \in [0,\lambda]$ and $v \in H^1_0(\Ur)$ satisfying
\begin{equation}
\label{eq:TwoScale}
(\mu^2 - \nabla \cdot a \nabla) v = (\mu^2 - \nabla \cdot \ab \nabla)\vb  \qquad \text{ in } \Ur,
\end{equation}
we have an $H^1$ estimate for the two-scale expansion $w$ associated to $\vb$ defined in \cref{eq:defTwoScale}

\begin{equation}\label{eq:TwoScaleH1}
\begin{split}
\Vert v - w \Vert_{\aH^1(\Ur)} & \leq   C(U,\Lambda,d) \left[ \vphantom{\frac{1}{2}} \Vert \vb \Vert_{\aH^2(\Ur)} + \left( \Vert \vb \Vert_{\aH^2(\Ur)} + \mu \Vert \vb \Vert_{\aH^1(\Ur)}\right) \X_1  \right. \\
& \qquad \left. + \left( \ell(\ec)^{\frac{1}{2}}\Vert \vb \Vert^{\frac{1}{2}}_{\aH^2(\Ur)}  \Vert \vb \Vert^{\frac{1}{2}}_{\aH^1(\Ur)} + \Vert \vb \Vert_{\aH^1(\Ur)} \right) \X_2   \right.\\
& \qquad \left. + \left( \ell(\ec)^{\frac{1}{2}} \left(\mu + \frac{1}{r}+ \frac{1}{\ell(\ec)}\right)\Vert \vb \Vert^{\frac{1}{2}}_{\aH^2(\Ur)}  \Vert \vb \Vert^{\frac{1}{2}}_{\aH^1(\Ur)} + \Vert \vb \Vert_{\aH^2(\Ur)} \right)\Y_1 \right].\\
\end{split}
\end{equation}
\end{theorem}
\begin{remark}
The explicit expression of $\X_1, \X_2, \Y_1$ can be checked in \cref{fig:factors}. They are the maximum of local spatial average of gradient and the flux of the first order corrector.
\end{remark}

\begin{proof} We give at first the proof of the deterministic part. We will see that the errors can finally be reduced to the estimates of two norms: the interior error term and a boundary layer term. The latter boundary term comes from the fact that $v$ and $w$ do not have the same boundary condition. So we introduce $b$ the solution of the equation
\begin{equation}
\label{eq:boundary}
\left\{
	\begin{array}{ll}
	(\mu^2 - \nabla \cdot \a \nabla) b = 0  & \text{ in } \Ur,\\
	b = \sum_{k=1}^d \p{k} (\vb \star \zeta) \phim_{e_k}  &\text{ on }\partial \Ur.
	\end{array}
\right.
\end{equation}
Then $(w - b)$ shares the same boundary condition as $v$. So, we have 
\begin{equation}
\label{eq:H1Decomposition}
\Vert v - w \Vert_{\aH^1(\Ur)} \leq \Vert v + b - w \Vert_{\aH^1(\Ur)} + \Vert b \Vert_{\aH^1(\Ur)}, 
\end{equation}
and we do estimates of the two parts respectively.

\begin{itemize}
\item \textbf{Estimate for $(v + b - w)$.}
We denote by $z:= v + b - w \in H^1_0(\Ur)$ and test it in \cref{eq:TwoScale} and \cref{eq:boundary}
\begin{eqnarray*}
\mu^2\int_{\Ur} z v  + \int_{\Ur} \nabla z \cdot \a \nabla v  &=& \mu^2\int_{\Ur} z \vb  + \int_{\Ur} \nabla z \cdot \ab \nabla \vb,   \\
\mu^2\int_{\Ur} z b  + \int_{\Ur} \nabla z \cdot \a \nabla b  &=& 0 .
\end{eqnarray*}
We do the sum to obtain that 
\begin{eqnarray*}
\mu^2\int_{\Ur} z (v + b)  + \int_{\Ur} \nabla z \cdot \a \nabla (v + b)  = \mu^2\int_{\Ur} z \vb  + \int_{\Ur} \nabla z \cdot \ab \nabla \vb.
\end{eqnarray*}
Using the fact $v + b = z + w$, we obtain 
$$
\mu^2\int_{\Ur}  |z|^2  + \int_{\Ur} \nabla z \cdot \a \nabla z  = \mu^2\int_{\Ur} z (\vb - w)  + \int_{\Ur} \nabla z \cdot (\ab \nabla \vb - \a \nabla w), 
$$
and we apply the uniform ellipticity condition to obtain
\begin{eqnarray*}
\mu^2 \Vert z \Vert^2_{\aL^2(\Ur)} + \Lambda^{-1}\Vert \nabla z \Vert^2_{\aL^2(\Ur)} &\leq & \mu^2 \Vert z \Vert_{\aL^2(\Ur)} \Vert w - \vb \Vert_{\aL^2(\Ur)} \\ 
& & + \Vert z \Vert_{\aH^1(\Ur)} \Vert \nabla \cdot \a \nabla w - \nabla \cdot \ab \nabla \vb \Vert_{\aH^{-1}(\Ur)}\\
(\text{Young's inequality})&\leq & \mu^2 \Vert z \Vert^2_{\aL^2(\Ur)} + \frac{\mu^2}{4}\Vert w - \vb \Vert^2_{\aL^2(\Ur)} \\
& & + \frac{\Lambda^{-1}}{2}\Vert z \Vert^2_{\aH^1(\Ur)} + \frac{\Lambda}{2} \Vert \nabla \cdot \a \nabla w - \nabla \cdot \ab \nabla \vb \Vert^2_{\aH^{-1}(\Ur)} \\
\Longrightarrow  \Vert \nabla z \Vert_{\aL^2(\Ur)} &\leq & \Lambda \Vert \nabla \cdot \a \nabla w - \nabla \cdot \ab \nabla \vb \Vert_{\aH^{-1}(\Ur)} + \sqrt{\Lambda}\mu \Vert w - \vb \Vert_{\aL^2(\Ur)} .
\end{eqnarray*}
We use Poincar\'e's inequality to conclude that 
\begin{equation}\label{eq:Clever}
\Vert z \Vert_{\aH^1(\Ur)} \leq  C(U) \left(\Lambda \Vert \nabla \cdot \a \nabla w - \nabla \cdot \ab \nabla \vb \Vert_{\aH^{-1}(\Ur)} + \sqrt{\Lambda}\mu \Vert w - \vb \Vert_{\aL^2(\Ur)}  \right).
\end{equation}

\item \textbf{Estimate for $b$.}
To estimate $b$ we use the property that it is the optimizer of the problem 
$$
\inf_{ \chi \in b + H^1_0(\Ur)} \mu^2\int_{\Ur} \chi^2 + \int_{\Ur} \nabla \chi \cdot \a \nabla \chi.
$$
So we give an upper bound of this functional by the following trial function 
\begin{equation}\label{eq:defTec}
T_{\ec} := \left(\Ind_{\Rd \backslash U_{r, 2 \ell(\ec)}} \star \zeta_{ \ell(\ec)} \right) \sum_{k=1}^d \p{k} (\vb \star \zeta) \phim_{e_k},
\end{equation}
where $U_{r, 2 \ell(\ec)}$ is defined as 
$$
U_{r, 2 \ell(\ec)} = \{x \in \Ur | d(x, \partial \Ur) > 2 \ell(\ec) \}.
$$
The motivation to test $T_{\ec}$ is the following: If we think the solution of elliptic equation is an average in some sense of the boundary value, then when the coefficient is oscillating, the boundary value is hard to propagate. So one naive candidate is just smoothing the boundary value in a small band of length $2 \ell(\ec)$.

By comparison, 
\begin{eqnarray*}
\mu^2\int_{\Ur} |b|^2  + \int_{\Ur} \nabla b \cdot \a \nabla b  &\leq & \mu^2\int_{\Ur} |T_{\ec}| ^2  + \int_{\Ur} \nabla T_{\ec}  \cdot \a \nabla T_{\ec}  \\
\Longrightarrow \Vert \nabla b \Vert_{\aL^2(\Ur)} &\leq & \mu \sqrt{\Lambda}\Vert T_{\ec} \Vert_{\aL^2(\Ur)} + \Lambda \Vert \nabla T_{\ec} \Vert_{\aL^2(\Ur)}.
\end{eqnarray*}

Moreover, to estimate the $L^2$ norm, we use once again Poincar\'e's inequality
\begin{eqnarray*}
\Vert b \Vert_{\aL^2(\Ur)} &=& \Vert b  - T_{\ec} + T_{\ec} \Vert_{\aL^2(\Ur)} \\
&\leq & \Vert b  - T_{\ec} \Vert_{\aL^2(\Ur)} + \Vert T_{\ec} \Vert_{\aL^2(\Ur)} \\
\text{(Poincar\'e's inequality)} &\leq & r \Vert \nabla(b  - T_{\ec}) \Vert_{\aL^2(\Ur)} + \Vert T_{\ec} \Vert_{\aL^2(\Ur)}\\
&\leq &  r\Vert \nabla b \Vert_{\aL^2(\Ur)} +  r\Vert \nabla T_{\ec} \Vert_{\aL^2(\Ur)} + \Vert T_{\ec} \Vert_{\aL^2(\Ur)}.
\end{eqnarray*}

We combine the two and get an estimate of $b$
\begin{eqnarray*}
\Vert b \Vert_{\aH^1(\Ur)} & = &\frac{1}{|\Ur|^{\frac{1}{d}}}\Vert b \Vert_{\aL^2(\Ur)} + \Vert \nabla b \Vert_{\aL^2(\Ur)}\\
& \leq & C(U)\left( \frac{1}{r}\Vert T_{\ec} \Vert_{\aL^2(\Ur)} + \mu \sqrt{\Lambda}\Vert T_{\ec} \Vert_{\aL^2(\Ur)} + (1+\Lambda) \Vert \nabla T_{\ec} \Vert_{\aL^2(\Ur)} \right).
\end{eqnarray*}
\end{itemize}
Finally, we put all the estimates above into \cref{eq:H1Decomposition}
\begin{equation}
\begin{split}\label{eq:H1Decomposition2}
\Vert v - w \Vert_{\aH^1(\Ur)} & \leq C(U, \Lambda)\left(  \vphantom{\frac{1}{2}} \Vert \nabla \cdot (\a \nabla w - \ab \nabla \vb) \Vert_{\aH^{-1}(\Ur)} + \mu \Vert w - \vb \Vert_{\aL^2(\Ur)} \right. \\
& \qquad  \qquad  \qquad \left. + \Vert \nabla T_{\ec} \Vert_{\aL^2(\Ur)} + \left(\frac{1}{r}+ \mu \right)\Vert T_{\ec} \Vert_{\aL^2(\Ur)} \right).
\end{split}
\end{equation}
To complete the proof of \Cref{thm:TwoScale}, we have to treat the random norms in \cref{eq:H1Decomposition2} respectively. It is the main task of the next section.
\end{proof}

\subsection{Construction of a vector field}
In this part, we analyze $\Vert \nabla \cdot (\a \nabla w - \ab \nabla \vb )\Vert_{\aH^{-1}(\Ur)} $ with the help of the \textit{flux corrector}.  Similar formulas appear both in \cite[Lemma 3.3]{armstrong2018iterative} and \cite[Chapter 6, Lemma 6.7]{armstrong2018quantitative}, here we give the version in our context.

At first, we introduce the \emph{flux corrector} $\S_p$.  For every $p \in \Rd$, since $\mathbf{g}_p := \a(p + \nabla \phi_p) - \ab p$ defines a divergence free field, i.e. $\nabla \cdot \mathbf{g}_p = 0 $, it admits a representation as the "curl" of some potential vector by Helmholtz's theorem. That is there exists a $\mathbb{R}^{d \times d}$ skew-symmetric matrix $\S_p$ such that
$$
\a(p + \nabla \phi_p) - \ab p = \nabla \cdot \S_p,
$$
where $\nabla \cdot \S_p $ is a $\Rd$ valued vector defined by $(\nabla \cdot \S_p)_i = \sum_{j=1}^d \p{j} \S_{p,ij}$. In order to "fix the gauge", for each $i,j \in \{1,2, \cdots d\}$, we force
$$
\Delta \S_{p,ij} = \p{j} \mathbf{g}_{p,i} - \p{i} \mathbf{g}_{p,j},
$$
and under this condition, $\S_p$ is unique up to a constant and $\{\nabla \S_p,ij, p \in \Rd, i,j \in \{1,2, \cdots d\}\}$ also forms a family of $\Zd$-stationary field like $\nabla \phi_p$. This quantity appears in the early work of periodic homogenization \cite[Lemma 3.1]{AveLin} and \cite[Lemma 4.4]{jikov2012homogenization}. See \cite[Lemma 6.1]{armstrong2018quantitative} for the well-definedness of $\S_p$ in stochatic homogenization setting and this quantity is also adapted in \cite[Lemma 1]{gloria2015quantification} as \emph{extended corrector}. We also define 
$$
\S_{p}^{(\ec)} = \S_{p} - \S_{p} \star \Phi_{\ec^{-1}},
$$
to truncate the constant part, so $\S_{p}^{(\ec)}$ is well-defined and $\Zd$-stationary.

We have the following identity.
\begin{lemma}
\label{lem:F}
For $\lambda > 0, \vb \in H_0^1(\Ur) \cap H^2(\Ur)$ and $w \in H^1(\Ur)$ as in \Cref{thm:TwoScale}. We construct a vector field $\Fi$ such that
$$
\nabla \cdot (\a \nabla w - \ab \nabla \vb) = \nabla \cdot \Fi,
$$
whose i-th component is given by 

\begin{equation}
\begin{split}\label{eq:F}
\Fi_i &= \sum_{j=1}^d (\a_{ij} - \ab_{ij}) \p{j} (\vb - \vb \star \zeta) + \sum_{j,k=1}^{d} \left(\a_{ij} \phim_{e_k} - \Sm_{e_k, ij}\right)\p{j}\p{k} (\vb \star \zeta) \\
& \qquad + \sum_{j,k=1}^{d} \left(\p{j}\S_{e_k,ij} \star \Phi_{\ec^{-1}} - \a_{ij}\p{j}\phi_{e_k} \star \Phi_{\ec^{-1}}\right) \p{k}(\vb \star \zeta).
\end{split}
\end{equation}
\end{lemma}
\begin{proof}
We develop
\begin{eqnarray*}
\left[\a \nabla w - \ab \nabla \vb \right]_i &=&  \left[ \a \nabla \left( \vb +  \sum_{k=1}^d \p{k} (\vb \star \zeta) \phim_{e_k} \right) - \ab \nabla \vb \right]_i \\
&=& \left[ (\a - \ab) \nabla \vb +  \a \nabla \sum_{k=1}^d \left(   \p{k} (\vb \star \zeta) \phim_{e_k} \right) \right]_i \\
&=& \underbrace{\left[ (\a - \ab) \nabla (\vb - \vb \star \zeta)\right]_i}_{\mathbf{I}} + \underbrace{\left[(\a - \ab)\nabla (\vb \star \zeta) + \a \nabla \sum_{k=1}^d \left(   \p{k} (\vb \star \zeta) \phim_{e_k} \right) \right]_i}_{\mathbf{II}}.
\end{eqnarray*}
The first term is indeed 
$$
\mathbf{I}=\left[(\a - \ab) \nabla (\vb - \vb \star \zeta)\right]_i = \sum_{j=1}^{d}(\a_{ij} - \ab_{ij}) \p{j}(\vb - \vb \star \zeta),
$$
as in the right hand side of the identity, so we continue to study the rest of the formula.
\begin{eqnarray*}
\mathbf{II}&=& \sum_{j=1}^{d}(\a_{ij} - \ab_{ij}) \p{j}(\vb \star \zeta) + \sum_{j,k=1}^d \a_{ij} \p{j}\left(   \p{k} (\vb \star \zeta) \phim_{e_k} \right)\\
&=& \sum_{j,k=1}^{d}(\a_{ij} - \ab_{ij}) \p{j}(\vb \star \zeta) \delta_{jk} + \sum_{j,k=1}^d  \a_{ij} \p{j} \p{k} (\vb \star \zeta) \phim_{e_k} + \sum_{j,k=1}^d  \a_{ij} \p{k} (\vb \star \zeta) \p{j}\phim_{e_k}\\
&=& \underbrace{\sum_{j,k=1}^{d}\left((\a_{ij} - \ab_{ij})\delta_{jk} + \a_{ij} \p{j}\phim_{e_k} \right) \p{k}(\vb \star \zeta)}_{\mathbf{II.1}}  + \underbrace{\sum_{j,k=1}^d  \a_{ij} \p{j} \p{k} (\vb \star \zeta) \phim_{e_k}}_{\mathbf{II.2}}.
\end{eqnarray*}
$\mathbf{II.2}$ appears in the right hand side of the formula, so it remains $\mathbf{II.1}$ to treat. We use the the definition of $\S_{e_k}^{(\ec)}$ in $\mathbf{II.1}$
\begin{eqnarray*}
\mathbf{II.1} &=& \sum_{k=1}^{d}\left[ \a(e_k + \nabla \phim_{e_k}) - \ab e_k   \right]_i \p{k}(\vb \star \zeta)\\
 &=& \sum_{k=1}^{d}\left[ \a(e_k + \nabla \phi_{e_k})- \ab e_k  - \a \nabla \phi_{e_k} \star \Phi_{\ec^{-1}}  \right]_i \p{k}(\vb \star \zeta)\\
 &=& \sum_{k=1}^{d}\left[ \nabla \cdot \S_{e_k}  - \a \nabla \phi_{e_k} \star \Phi_{\ec^{-1}} \right]_i \p{k}(\vb \star \zeta)\\
 &=& \underbrace{\sum_{k=1}^d \left[ \nabla \cdot \Sm_{e_k} \right]_i \p{k}(\vb \star \zeta)}_{\mathbf{III}} + \sum_{k=1}^{d} \left[\nabla \cdot \S_{e_k} \star \Phi_{\ec^{-1}} - \a\nabla \phi_{e_k} \star \Phi_{\ec^{-1}} \right]_i \p{k}(\vb \star \zeta).
\end{eqnarray*}
All the terms match well except $\mathbf{III}$, where we have to look for an equal form after divergence. Thanks to the property of skew-symmetry, we have
\begin{eqnarray*}
\nabla \cdot \mathbf{III} &=& \nabla \cdot \left( \sum_{k=1}^d \left[ \nabla \cdot \Sm_{e_k} \right]_i \p{k}(\vb \star \zeta) \right) \\
&=& \sum_{i,j,k=1}^d \p{i}\left(\p{j}\Sm_{e_k,ij} \, \p{k}(\vb \star \zeta) \right)\\
(\text{Integration by parts}) &=& \sum_{i,j,k=1}^d \p{i}\p{j}\left(\Sm_{e_k,ij}  \p{k}(\vb \star \zeta) \right) - \p{i}\left(\Sm_{e_k,ij} \, \p{j}\p{k}(\vb \star \zeta) \right)\\
(\text{Skew-symmetry of }\S)&=& - \nabla \cdot \left(\sum_{j,k=1}^d \Sm_{e_k,ij}  \, \p{j}\p{k}(\vb \star \zeta) \right). 
\end{eqnarray*}
This finishes the proof.
\end{proof}

\subsection{Quantitative description of $\phim_{e_k}$ and $\Sm_{e_k}$}
In this subsection, we will give some quantitative descriptions of $\phim_{e_k}$ and $\Sm_{e_k}$, which serve as the bricks to form $\X_1, \X_2, \Y_1$.
\begin{lemma}[Estimate of corrector]
\label{lem:Cell}
For each $s \in (0,2)$, there exists a constant $0 < C(s,\Lambda, d) < \infty$ such that for every $\ec \in (0,1)$, $i,j,k \in \{1, \cdots d\}, z \in \mathbb{Z}^d$
\begin{equation}
\label{eq:CellEnergy1}
\begin{split}
\Vert \nabla \phi_{e_k} \star \Phi_{\ec^{-1}} \Vert_{\aL^2(z + \cu)} \leq \GO_s(C \ec^{\frac{d}{2}}),  & \quad \Vert \nabla \S_{e_k,ij} \star \Phi_{\ec^{-1}} \Vert_{\aL^2(z + \cu)} \leq \GO_s(C \ec^{\frac{d}{2}}), \\
\Vert \phim_{e_k} \Vert_{\aL^2(z +  \cu)} \leq \GO_s(C \ell(\ec)),  & \quad \Vert \Sm_{e_k, ij} \Vert_{\aL^2(z + \cu)} \leq \GO_s(C \ell(\ec)).
\end{split}
\end{equation}
\end{lemma}
\begin{proof}
We study at first the part $\phim_{e_k}$. \cite[Theorem 4.1]{armstrong2018quantitative}  gives us three useful estimates

\begin{itemize}
\item $d \geq 2, r > 1$, for any $x \in \Rd$
\begin{equation}\label{eq:CellCalssical1}
  |\nabla \phi_{e_k} \star \Phi_{r}(x)| \leq  \GO_s(C(s,d,\Lambda) r^{-\frac{d}{2}}).  
\end{equation}
\item $d \geq3,$
\begin{equation}\label{eq:CellCalssical2}
\Vert \phi_{e_k} \Vert_{\aL^2(\cu)} \leq  \GO_{2}(C(d, \Lambda)). 
\end{equation}
\item $d = 2$, for any $2 \leq r < R < \infty$, and $x,y\in \Rd,$ 
\begin{subequations}
\begin{empheq}{align}
\Vert \phi_{e_k} - \phi_{e_k} \star \Phi_{r}(0) \Vert_{\aL^2(\cu_r)} &\leq  \GO_s(C(s,\Lambda) \log^{\frac{1}{2}}r),  \label{eq:CellCalssical3}\\
|(\phi_{e_k} * \Phi_r)(x) - (\phi_{e_k} * \Phi_R)(y)| &\leq  \GO_s\left(C(s,\Lambda) \log^{\frac{1}{2}} \left( 2 + \frac{R + |x-y|}{r}\right) \right).  \label{eq:CellCalssical4}
\end{empheq}
\end{subequations}
\end{itemize}
Informally speaking, the behavior of corrector when $d \geq 3$ is of size constant, but has a logarithm increment when $d = 2$ and this explicates why we have $\ell(\ec)$ in \cref{eq:CellEnergy1}.

\begin{enumerate}
\item{Proof of $\Vert \nabla \phi_{e_k} \star \Phi_{\ec^{-1}} \Vert_{\aL^2(z + \cu)} \leq \GO_s(C \ec^{\frac{d}{2}}).$}

By choosing $r = \ec^{-1}$ in \cref{eq:CellCalssical1} and using \cref{eq:OSum}, we have
\begin{eqnarray*}
 |\nabla \phi_{e_k} \star \Phi_{\ec^{-1}}(x)| &\leq & \GO_{s}(C \ec^{\frac{d}{2}})\\
\Longrightarrow  |\nabla \phi_{e_k} \star \Phi_{\ec^{-1}}(x)|^2  &\leq & \GO_{s/2}(C^2 \ec^{d}) \\
\Longrightarrow  \int_{z + \cu} |\nabla \phi_{e_k} \star \Phi_{\ec^{-1}}(x)|^2 \, dx &\leq & \GO_{s/2}(C^2 \ec^{d}) \\
\Longrightarrow \Vert \nabla \phi_{e_k} \star \Phi_{\ec^{-1}} \Vert_{\aL^2(z + \cu)} &\leq &\GO_s(C \ec^{\frac{d}{2}}).
\end{eqnarray*}

\item{Proof that if $d \geq 3,$ then $\Vert \phim_{e_k} \Vert_{\aL^2(z +  \cu)} \leq \GO_s(C).$}
We apply \cref{eq:CellCalssical2} to get that 
\begin{eqnarray*}
\Vert \phim_{e_k} \Vert_{\aL^2(z +  \cu)} \leq  \Vert \phi_{e_k} \Vert_{\aL^2(z +  \cu)} +  \Vert  \phi_{e_k} \star \Phi_{\ec^{-1}} \Vert_{\aL^2(z +  \cu)},
\end{eqnarray*}
where the first one comes from a modified version of \cref{eq:CellCalssical2}. In fact, by the $\Zd$-stationarity of $\phi$ \Cref{def:Corrector}, we have $\Vert \phi_{e_k} \Vert_{\aL^2(z + \cu)} \leq  \GO_{2}(C(d, \Lambda))$, then for every $s \in (0,2)$, we set $C(s, d,\Lambda) = 3^{\frac{1}{s}}C(d, \Lambda)$, then
\begin{eqnarray*}
\Expt\left[\exp\left(\left(\frac{\Vert \phi_{e_k} \Vert_{\aL^2(z + \cu)}}{C(s, d,\Lambda)}\right)^s_+\right)\right] 
&=&  \Expt\left[\exp\left(\frac{1}{3}\left(\frac{\Vert \phi_{e_k} \Vert_{\aL^2(z + \cu)}}{C(d, \Lambda)}\right)^s\right)\right] \\
&=&\Expt\left[\left(\exp\left(\left(\frac{\Vert \phi_{e_k} \Vert_{\aL^2(z + \cu)}}{C(d, \Lambda)}\right)^s\right)\right)^{\frac{1}{3}}\right] \\ 
(\text{Jensen's inequality})&\leq & \left(\Expt\left[\exp\left(\left(\frac{\Vert \phi_{e_k} \Vert_{\aL^2(z + \cu)}}{C(d, \Lambda)}\right)^s\right)\right]\right)^{\frac{1}{3}}.
\end{eqnarray*}
We decompose the last term into two parts
\begin{eqnarray*}
& & \Expt\left[\exp\left(\left(\frac{\Vert \phi_{e_k} \Vert_{\aL^2(z + \cu)}}{C(s, d,\Lambda)}\right)^s_+\right)\right]\\
& \leq & \left(\Expt\left[\exp\left(\left(\frac{\Vert \phi_{e_k} \Vert_{\aL^2(z + \cu)}}{C(d, \Lambda)}\right)^s\right)\Ind_{\{\Vert \phi_{e_k} \Vert_{\aL^2(z + \cu)} \leq C(d, \Lambda)\}}\right] \right. \\
& & \qquad + \left. \Expt\left[\exp\left(\left(\frac{\Vert \phi_{e_k} \Vert_{\aL^2(z + \cu)}}{C(d, \Lambda)}\right)^s\right)\Ind_{\{\Vert \phi_{e_k} \Vert_{\aL^2(z + \cu)} \geq C(d, \Lambda)\}}\right] \right)^{\frac{1}{3}} \\ 
& \leq & \left(\underbrace{\Expt\left[\exp\left(\left(\frac{\Vert \phi_{e_k} \Vert_{\aL^2(z + \cu)}}{C(d, \Lambda)}\right)^s\right)\Ind_{\{\Vert \phi_{e_k} \Vert_{\aL^2(z + \cu)} \leq C(d, \Lambda)\}}\right]}_{\leq e} \right. \\
& & \qquad + \left. \underbrace{\Expt\left[\exp\left(\left(\frac{\Vert \phi_{e_k} \Vert_{\aL^2(z + \cu)}}{C(d, \Lambda)}\right)^2\right)\Ind_{\{\Vert \phi_{e_k} \Vert_{\aL^2(z + \cu)} \geq C(d, \Lambda)\}}\right]}_{\leq 2} \right)^{\frac{1}{3}} \\ 
&\leq  &(e + 2)^{\frac{1}{3}} \leq  2.
\end{eqnarray*}
Thus we prove that for any $s \in (0,2), C(s, d,\Lambda) = 3^{\frac{1}{s}}C(d, \Lambda)$
\begin{equation}\label{eq:O2toOs}
\Vert \phi_{e_k} \Vert_{\aL^2(z + \cu)} \leq  \GO_{2}(C( d, \Lambda)) \Longrightarrow  \Vert \phi_{e_k} \Vert_{\aL^2(z + \cu)} \leq  \GO_{s}(C(s, d, \Lambda)). 
\end{equation}
We focus on the second one that 
\begin{eqnarray*}
\Vert  \phi_{e_k} \star \Phi_{\ec^{-1}} \Vert^2_{\aL^2(z +  \cu)} &=& \int_{z+\cu} \left| \int_{\Rd} \phi_{e_k}(x-y) \Phi_{\ec^{-1}}(y) \, dy \right|^2 \, dx \\
&=& \int_{z+\cu} \left| \int_{\Rd} \phi_{e_k}(x-y) \Phi^{\frac{1}{2}}_{\ec^{-1}}(y)  \Phi^{\frac{1}{2}}_{\ec^{-1}}(y) \, dy \right|^2 \, dx \\
(\text{H\"older's inequality})&\leq & \int_{z+\cu}  \left( \int_{\Rd} \phi^2_{e_k}(x-y) \Phi_{\ec^{-1}}(y)  \, dy \right) \underbrace{\left( \int_{\Rd} \Phi_{\ec^{-1}}(y) \, dy \right)}_{=1} \, dx \\
&=& \int_{z+\cu}  \int_{\Rd} \phi^2_{e_k}(x-y) \Phi_{\ec^{-1}}(y)  \, dy  \, dx \\
(\cref{eq:OSum}) &\leq & \GO_s(C).
\end{eqnarray*}
In the last step, we treat $\Phi_{\ec^{-1}}$ as a weight for different small cubes so we could apply \cref{eq:OSum} and \cref{eq:O2toOs}.

\item{Proof that if $d = 2$, then $\Vert \phim_{e_k} \Vert_{\aL^2(z +  \cu)} \leq \GO_s(C \ell(\ec)).$}

This part is a little more difficult than the case $d \geq 3$ since we have only \cref{eq:CellCalssical3} and \cref{eq:CellCalssical4} instead of \cref{eq:CellCalssical2} when $d=2$. This foreces us to do more steps of difference.

We apply \cref{eq:CellCalssical3} \cref{eq:CellCalssical4} to $\Vert \phim_{e_k} \Vert_{\aL^2(z + \cu)} $ for any $\ec \in (0,\frac{1}{2}]$
\begin{eqnarray*}
\Vert \phim_{e_k} \Vert_{\aL^2(z + \cu)} &=& \Vert \phi_{e_k} - \phi_{e_k} \star \Phi_{\ec^{-1}} \Vert_{\aL^2(z + \cu)} \\
&\leq & \Vert \phi_{e_k} - \phi_{e_k} \star \Phi_2(z) \Vert_{\aL^2(z + \cu)} + \Vert \phi_{e_k} \star \Phi_2(z) - \phi_{e_k} \star \Phi_2 \Vert_{\aL^2(z + \cu)}\\
& & + \Vert \phi_{e_k} \star \Phi_2 - \phi_{e_k} \star \Phi_{\ec^{-1}} \Vert_{\aL^2(z + \cu)}\\
&\leq & 4 \underbrace{\Vert \phi_{e_k} - \phi_{e_k} \star \Phi_2(z) \Vert_{\aL^2(B_2(z))}}_{\text{Apply } \cref{eq:CellCalssical3}} + 4 \underbrace{\Vert \phi_{e_k} \star \Phi_2(z) - \phi_{e_k} \star \Phi_2 \Vert_{\aL^2(B_2(z))}}_{\text{Apply }\cref{eq:CellCalssical4}} \\
& & + \underbrace{\Vert \phi_{e_k} \star \Phi_2 - \phi_{e_k} \star \Phi_{\ec^{-1}} \Vert_{\aL^2(z + \cu)}}_{\text{Apply } \cref{eq:CellCalssical4} }\\
&\leq & \GO_s(C) + \GO_s(C) + \GO_s(C\log^{\frac{1}{2}}(2 + (2\ec)^{-1}))\\
(\cref{eq:OSum})&\leq & \GO_s(C \ell(\ec)). 
\end{eqnarray*}
Here we use $\Phi_2$ since \cref{eq:CellCalssical3} and \cref{eq:CellCalssical4} require that the scale should be bigger than $2$. In last step, we use also the condition $\ec \leq \frac{1}{2}$ to give up the constant term.
\end{enumerate}

Since $\S_{e_k}$ has the same type of estimate \cref{eq:CellCalssical1},\cref{eq:CellCalssical2},\cref{eq:CellCalssical3},\cref{eq:CellCalssical4} as $\phi_{e_k}$, see \cite[Proposition 6.2]{armstrong2018quantitative}, we apply the same procedure to obtain the other half of the \cref{eq:CellEnergy1}.
\end{proof}

\subsection{Detailed $H^{-1}$ and boundary layer estimate}
\label{sec:Details}
In this subection, we complete the proof of \Cref{thm:TwoScale}, which remains to give an explicit random variable in the formula \cref{eq:H1Decomposition2}. This requires to analyze several norms like $\Vert \nabla \cdot ( \a \nabla w - \ab \nabla \vb )\Vert_{\aH^{-1}(\Ur)},  \Vert w - \vb \Vert_{\aL^2(\Ur)},  \Vert \nabla T_{\ec} \Vert_{\aL^2(\Ur)}, \Vert T_{\ec} \Vert_{\aL^2(\Ur)}$. We will make the use of two technical lemmas in \Cref{sec:improve} and \Cref{lem:Cell} to estimate them.

\subsubsection{Estimate of $\Vert \nabla \cdot (\a \nabla w - \ab \nabla \vb )\Vert_{\aH^{-1}(\Ur)}$}
With the help of \Cref{lem:F}, we have 
$$
\Vert \nabla \cdot ( \a \nabla w - \ab \nabla \vb) \Vert_{\aH^{-1}(\Ur)} = \Vert \nabla \cdot \Fi \Vert_{\aH^{-1}(\Ur)} \leq \Vert \Fi \Vert_{\aL^2(\Ur)},
$$
and we use the identity in \cref{eq:F} to obtain 
\begin{eqnarray*}
\Vert \Fi \Vert_{\aL^2(\Ur)} &\leq & \underbrace{\sum_{j=1}^d \Vert (\a - \ab) \nabla (\vb - \vb \star \zeta) \Vert_{\aL^2(\Ur)}}_{\mathbf{H.1}} \\
& &  + \underbrace{\sum_{j,k=1}^{d} \left\Vert \left(\a \phim_{e_k} - \Sm_{e_k}\right) \p{j}\p{k} (\vb \star \zeta) \right\Vert_{\aL^2(\Ur)}}_{\mathbf{H.2}} \\
& & + \underbrace{\sum_{k=1}^{d} \left\Vert  \left(\nabla \S_{e_k} \star \Phi_{\ec^{-1}} - \a \nabla \phi_{e_k} \star \Phi_{\ec^{-1}}\right) \p{k}(\vb \star \zeta)\right\Vert_{\aL^2(\Ur)}}_{\mathbf{H.3}}.
\end{eqnarray*}
We treat the three terms respectively.
For $\mathbf{H.1}$, we recall that $\vb \star \zeta$ means $\Ext(\vb) \star \zeta$ and use the approximation of identity, see for example \cite[Lemma 6.8]{armstrong2018quantitative} 
\begin{equation*}
\mathbf{H.1} \leq \frac{d \Lambda}{\vert \Ur \vert^{\frac{d}{2}}} \Vert \nabla \Ext(\vb) - \nabla \Ext(\vb) \star \zeta \Vert_{L^2(\Rd)} \leq \frac{d \Lambda}{\vert \Ur \vert^{\frac{d}{2}}} \Vert \nabla^2 \Ext(\vb) \Vert_{L^2(\Rd)}.
\end{equation*}
We recall the estimate \cref{eq:Extension} that  
\begin{equation*}
\frac{d \Lambda}{\vert \Ur \vert^{\frac{d}{2}}} \Vert \nabla^2 \Ext(\vb) \Vert_{L^2(\Rd)} \leq C(U,\Lambda, d) \Vert \vb \Vert_{\aH^2(\Ur)},
\end{equation*}
so we get $\mathbf{H.1}  \leq C(U,\Lambda, d) \Vert \vb \Vert_{\aH^2(\Ur)}$.

For $\mathbf{H.2}$, since $\Vert \phim_{e_k}\Vert_{\aL^2(z + \cu)},  \Vert \Sm_{e_k} \Vert_{\aL^2(z + \cu)}$ are obtained in \Cref{lem:Cell}, we could use the \Cref{lem:MixedNorm} where we treat the cell of the scale $\epsilon = 1$ and take $g = \p{j}\p{k} \Ext(\vb)$, and $f = (\a \phim_{e_k} - \Sm_{e_k})$
\begin{eqnarray*}
\mathbf{H.2} &=& \sum_{j,k=1}^{d} \Vert \left(\a \phim_{e_k} - \Sm_{e_k}\right) ( \p{j}\p{k}\Ext(\vb) \star \zeta) \Vert_{\aL^2(\Ur)} \\
 &\leq & \frac{C(\Lambda, d)}{\vert \Ur \vert^{\frac{d}{2}}} \sum_{k=1}^d  \max_{z \in \Zd \cap (\Ur+\cu)} \left(\Vert  \phim_{e_k} \Vert_{\aL^2(z + \cu)}  +  \Vert \Sm_{e_k} \Vert_{\aL^2(z + \cu)} \right) \Vert  \nabla^2 \Ext(\vb) \Vert_{L^2(\Ur + 3\cu)}.
\end{eqnarray*}
Once again we apply the Sobolev extension estimate \cref{eq:Extension} that 
\begin{eqnarray*}
\vert \Ur \vert^{-\frac{d}{2}} \Vert  \nabla^2 \Ext(\vb) \Vert_{L^2(\Ur + 3\cu)} \leq \vert \Ur \vert^{-\frac{d}{2}} \Vert  \nabla^2 \Ext(\vb) \Vert_{L^2(\Rd)} \leq  C(U, d) \Vert \vb \Vert_{\aH^2(\Ur)}.
\end{eqnarray*}
We extract the term of random variable 
\begin{equation}
\label{eq:X1}
\X_1 := \sum_{k=1}^d \max_{z \in \Zd \cap (\Ur+\cu)} \left(\Vert \phim_{e_k}\Vert_{\aL^2(z + \cu)}  +  \Vert \Sm_{e_k} \Vert_{\aL^2(z + \cu)} \right),
\end{equation}
and obtain that $\mathbf{H.2} \leq C(U, \Lambda, d) \X_1 \Vert \vb \Vert_{\aH^2(\Ur)}$. Moreover, \Cref{lem:max} and \Cref{lem:Cell} can be applied here to estimate the size of random variables that  
$$
\X_1 \leq \GO_s\left(C(U,s,d) \ell(\ec)(\log r)^{\frac{1}{s}} \right). 
$$

The above estimation gives a good recipe for the remaining part. For $\mathbf{H.3}$, we have 
\begin{eqnarray*}
\mathbf{H.3} &=& \sum_{k=1}^{d} \Vert  \left(\nabla \S_{e_k} \star \Phi_{\ec^{-1}} - \a \nabla \phi_{e_k} \star \Phi_{\ec^{-1}}\right) (\p{k}\Ext(\vb) \star \zeta)\Vert_{\aL^2(\Ur)} \\
&\leq & C(U, \Lambda, d) \X_2 \Vert \vb \Vert_{\aH^1(\Ur)},
\end{eqnarray*}
where we extract that 
\begin{equation}
\label{eq:X2}
\X_2 := \sum_{k=1}^d  \max_{z \in  \Zd \cap (\Ur + \cu)} \left(\Vert 
\nabla \phi_{e_k} \star \Phi_{\ec^{-1}} \Vert_{\aL^2(z + \cu)}  +  \Vert \nabla \S_{e_k} \star \Phi_{\ec^{-1}} \Vert_{\aL^2(z + \cu)} \right),
\end{equation}
and we apply \Cref{lem:max} and \Cref{lem:Cell} to get
$$
\X_2 \leq \GO_s\left(C(U,s,d) \ec^{\frac{d}{2}}(\log r)^{\frac{1}{s}} \right).
$$

Combing $\mathbf{H.1}, \mathbf{H.2}, \mathbf{H.3}$, we get 
\begin{equation}\label{eq:TwoScale1}
\boxed{\Vert \nabla \cdot (\a \nabla w - \ab \nabla \vb) \Vert_{\aH^{-1}(\Ur)} \leq C(U, \Lambda, d)\left(\Vert \vb \Vert_{\aH^2(\Ur)} + \Vert \vb \Vert_{\aH^2(\Ur)} \X_1, + \Vert \vb \Vert_{\aH^1(\Ur)} \X_2 \right)}.
\end{equation}

\subsubsection{Estimate of $\Vert w - \vb \Vert_{\aL^2(\Ur)}$}
For $\Vert w - \vb \Vert_{\aL^2(\Ur)}$, we use \Cref{lem:MixedNorm}  and \cref{eq:Extension} to obtain that
\begin{eqnarray*}
\Vert w - \vb \Vert_{\aL^2(\Ur)} &=& \Vert \sum_{k=1}^d \phim_{e_k} \p{k} (\Ext(\vb) \star \zeta)\Vert_{\aL^2(\Ur)} \\
&\leq & C(U, \Lambda, d)\sum_{k=1}^d  \max_{z \in \Zd \cap (\Ur+\cu)} \left(\Vert  \phim_{e_k}  \Vert_{\aL^2(z + \cu)}  \right) \Vert \vb \Vert_{\aH^1(\Ur)}\\
&\leq & C(U, \Lambda, d) \Vert \vb \Vert_{\aH^1(\Ur)} \X_1.
\end{eqnarray*}
\begin{equation}\label{eq:TwoScale2}
\Longrightarrow \boxed{\Vert w - \vb \Vert_{\aL^2(\Ur)} \leq C(U, \Lambda, d) \Vert \vb \Vert_{\aH^1(\Ur)} \X_1}.
\end{equation}

\subsubsection{Estimate of $\Vert \nabla T_{\ec} \Vert_{\aL^2(\Ur)},\Vert T_{\ec} \Vert_{\aL^2(\Ur)}$}
Finally, we come to the estimate of $\Vert \nabla T_{\ec} \Vert_{\aL^2(\Ur)},\Vert T_{\ec} \Vert_{\aL^2(\Ur)}$.  We study  $\Vert \nabla T_{\ec} \Vert_{\aL^2(\Ur)}$ at first.
\begin{eqnarray*}
\Vert \nabla T_{\ec} \Vert_{\aL^2(\Ur)} &=& \underbrace{\left\Vert \left(\Ind_{\Rd \backslash U_{r, 2l(\ec)}} \star \frac{1}{\ell^{\frac{d}{2}+1}(\ec)} (\nabla \zeta)\left(\frac{\cdot}{\ell(\ec)}\right) \right) \sum_{k=1}^d \p{k} (\vb \star \zeta) \phim_{e_k} \right\Vert_{\aL^2(\Ur)}}_{\mathbf{T.1}} \\
& & + \underbrace{\left\Vert  \left(\Ind_{\Rd \backslash U_{r, 2\ell(\ec)}} \star \zeta_{\ell(\ec)} \right) \sum_{k=1}^d \p{k} ( \nabla \vb \star \zeta) \phim_{e_k} \right\Vert_{\aL^2(\Ur)}}_{\mathbf{T.2}} \\
& & + \underbrace{\left \Vert  \left(\Ind_{\Rd \backslash U_{r, 2\ell(\ec)}} \star \zeta_{\ell(\ec)} \right) \sum_{k=1}^d \p{k} (\vb \star \zeta) \nabla \phim_{e_k} \right \Vert_{\aL^2(\Ur)}.}_{\mathbf{T.3}}
\end{eqnarray*}
For the term $\mathbf{T.1}$, we use \cref{eq:Extension} and \cref{eq:MixedNorm} that
\begin{eqnarray*}
\mathbf{T.1} &\leq & C \frac{1}{\ell(\ec)} \left\Vert  \sum_{k=1}^d \Ind_{\Ur \backslash U_{r, 2\ell(\ec)}} \p{k}(\Ext(\vb) \star \zeta) \phim_{e_k} \right \Vert_{\aL^2(\Ur)} \\
&\leq & \frac{C(U, d)}{\ell(\ec)}\sum_{k=1}^d  \max_{z \in  \Zd \cap (\Ur \backslash U_{r, 2\ell(\ec)}+\cu)} \left(\Vert  \phim_{e_k}  \Vert_{\aL^2(z +  \cu)}  \right) \Vert  \Ind_{\Ur \backslash U_{r, 2\ell(\ec)}} \vb \Vert_{\aH^1(\Ur)}.
\end{eqnarray*}
Here we should pay attention to one small improvement: The domain of integration is in fact restricted in $\Ur \backslash U_{r, 2\ell(\ec)}$, so we would like to give it a bound in terms of $\aH^2(\Ur)$ rather than $\aH^1(\Ur)$. We borrow a trace estimate in \cite[Proposition A.1]{armstrong2018iterative} that for $f\in H^1(\Ur)$
\begin{equation}
\label{eq:Trace}
\Vert f \Ind_{\Ur \backslash U_{r, 2 \ell(\ec)}} \Vert_{\aL^2(\Ur)} \leq C(U,d) \ell(\ec)^{\frac{1}{2}}
 \Vert f  \Vert^{\frac{1}{2}}_{\aH^1(\Ur)} \Vert f  \Vert^{\frac{1}{2}}_{\aL^2(\Ur)},
\end{equation} 
using \cref{eq:Trace} then we obtain an estimate
\begin{eqnarray*}
\mathbf{T.1} &\leq & \frac{C(U, d)}{\ell^{\frac{1}{2}}(\ec)}\sum_{k=1}^d  \max_{z \in \Zd \cap (\Ur \backslash U_{r, 2\ell(\ec)} + \cu)} \left(\Vert  \phim_{e_k}  \Vert_{\aL^2(z + \cu)}  \right) \Vert \vb \Vert^{\frac{1}{2}}_{\aH^2(\Ur)}  \Vert \vb \Vert^{\frac{1}{2}}_{\aH^1(\Ur)},
\end{eqnarray*}
so we define the random variable
\begin{equation}
\label{eq:Y1}
\Y_1  := \sum_{k=1}^d \max_{z \in \Zd \cap (\Ur \backslash U_{r, 2\ell(\ec)} + \cu)} \left(\Vert  \phim_{e_k}  \Vert_{\aL^2(z + \cu)} \right),
\end{equation}
and we have the estimate by \Cref{lem:max} and \Cref{lem:Cell}
$$
\Y_1 \leq \GO_s\left(C(U,s,d) \ell(\ec)(\log r)^{\frac{1}{s}} \right).
$$

We skip the details since they are analogue to the previous part. $\mathbf{T.3}$ follows from the same type of estimate as $\mathbf{T.1}$ and $\mathbf{T.2}$ is routine where we suffices to apply \Cref{lem:MixedNorm} and \cref{eq:Trace}. We find that
\begin{eqnarray*}
\mathbf{T.2} &\leq & C(U, d)  \Vert \vb \Vert_{\aH^2(\Ur)} \Y_1, \\
\mathbf{T.3} &\leq & C(U,d) \Vert \vb \Vert^{\frac{1}{2}}_{\aH^2(\Ur)} \Vert \vb \Vert^{\frac{1}{2}}_{\aH^1(\Ur)} \ell^{\frac{1}{2}} (\ec)\X_2.
\end{eqnarray*}

The three estimates of $\mathbf{T.1}, \mathbf{T.2}, \mathbf{T.3}$ implies that 
\begin{subequations}\label{eq:TwoScale3}
\begin{empheq}[box=\widefbox]{align}
\Vert \nabla T_{\ec} \Vert_{\aL^2(\Ur)} &\leq  C(U, d) \Vert \vb \Vert^{\frac{1}{2}}_{\aH^2(\Ur)}  \Vert \vb \Vert^{\frac{1}{2}}_{\aH^1(\Ur)}  \frac{1}{\ell^{\frac{1}{2}}(\ec)} \Y_1 \\
& + C(U,d)\left(  \Vert \vb \Vert_{\aH^2(\Ur)} \Y_1 +\Vert \vb \Vert^{\frac{1}{2}}_{\aH^2(\Ur)} \Vert \vb \Vert^{\frac{1}{2}}_{\aH^1(\Ur)} \ell^{\frac{1}{2}} (\ec)\X_2\right).
\end{empheq}
\end{subequations}

Finally, we find that $\Vert T_{\ec} \Vert_{\aL^2(\Ur)}$ has been contained in the estimate $\mathbf{T.1}$ that 
\begin{equation}\label{eq:TwoScale4}
\boxed{\Vert T_{\ec} \Vert_{\aL^2(\Ur)} \leq C(U,d)  \Vert \vb \Vert^{\frac{1}{2}}_{\aH^2(\Ur)} \Vert \vb \Vert^{\frac{1}{2}}_{\aH^1(\Ur)}\ell^{\frac{1}{2}} (\ec)\Y_1}.
\end{equation}

\cref{eq:H1Decomposition2}, \cref{eq:TwoScale1}, \cref{eq:TwoScale2}, \cref{eq:TwoScale3}, \cref{eq:TwoScale4} conclude the proof of \Cref{thm:TwoScale}. We have 
\begin{eqnarray*}
\Vert v - w \Vert_{\aH^1(\Ur)} &\leq &  C(U, \Lambda)\left( \vphantom{\frac{1}{2}}\Vert \a \nabla w - \ab \nabla \vb \Vert_{\aH^{-1}(\Ur)} + \mu \Vert w - \vb \Vert_{\aL^2(\Ur)} \right. \\
 & & \qquad\left. + \Vert \nabla T_{\ec} \Vert_{\aL^2(\Ur)} + (\frac{1}{r}+ \mu )\Vert T_{\ec} \Vert_{\aL^2(\Ur)} \right)\\
& \leq & C(U,\Lambda,d) \left( \vphantom{\frac{1}{2}} \Vert \vb \Vert_{\aH^2(\Ur)} + \Vert \vb \Vert_{\aH^2(\Ur)} \X_1 + \Vert \vb \Vert_{\aH^1(\Ur)} \X_2 + \mu \Vert \vb \Vert_{\aH^1(\Ur)} \X_1 \right. \nonumber \\
& & \qquad \left. + \ell(\ec)^{\frac{1}{2}}\Vert \vb \Vert^{\frac{1}{2}}_{\aH^2(\Ur)}  \Vert \vb \Vert^{\frac{1}{2}}_{\aH^1(\Ur)}\left(\left(\mu + \frac{1}{r}+ \frac{1}{\ell(\ec)}\right) \Y_1  + \X_2 \right) + \Vert \vb \Vert_{\aH^2(\Ur)} \Y_1 \right) \\
&=&  C(U,\Lambda,d) \left[ \vphantom{\frac{1}{2}} \Vert \vb \Vert_{\aH^2(\Ur)} + \left( \Vert \vb \Vert_{\aH^2(\Ur)} + \mu \Vert \vb \Vert_{\aH^1(\Ur)}\right) \X_1  \right. \nonumber \\
& & \qquad \left. + \left( \ell(\ec)^{\frac{1}{2}}\Vert \vb \Vert^{\frac{1}{2}}_{\aH^2(\Ur)}  \Vert \vb \Vert^{\frac{1}{2}}_{\aH^1(\Ur)} + \Vert \vb \Vert_{\aH^1(\Ur)} \right) \X_2   \right.\\
& & \qquad \left. + \left( \ell(\ec)^{\frac{1}{2}} \left(\mu + \frac{1}{r}+ \frac{1}{\ell(\ec)}\right)\Vert \vb \Vert^{\frac{1}{2}}_{\aH^2(\Ur)}  \Vert \vb \Vert^{\frac{1}{2}}_{\aH^1(\Ur)} + \Vert \vb \Vert_{\aH^2(\Ur)}\right) \Y_1   \right].
\end{eqnarray*}

We add one table of $\X_1, \X_2, \Y_1$ to check the its typical size.

\begin{figure}[h!]
\begin{adjustwidth}{}{}
\begin{tabular}{|c | c | c |} 
 \hline
 R.V & Expression & $\GO_s$ size \\ [0.5ex] 
 \hline
 $\X_1$ & $\sum_{k=1}^d \max\limits_{z \in \Zd \cap (\Ur+\cu)} \left(\Vert \phim_{e_k}\Vert_{\aL^2(z + \cu)}  +  \Vert \Sm_{e_k} \Vert_{\aL^2(z + \cu)} \right)$ & $\GO_s\left(C \ell(\ec)(\log r)^{\frac{1}{s}} \right)$ \\ 
 $\X_2$ & $\sum_{k=1}^d  \max\limits_{z \in  \Zd \cap (\Ur + \cu)} \left(\Vert 
\nabla \phi_{e_k} \star \Phi_{\ec^{-1}} \Vert_{\aL^2(z + \cu)}  +  \Vert \nabla \S_{e_k} \star \Phi_{\ec^{-1}} \Vert_{\aL^2(z + \cu)} \right)$ & $\GO_s\left(C \ec^{\frac{d}{2}}(\log r)^{\frac{1}{s}} \right)$ \\
 $\Y_1$ & $\sum_{k=1}^d \max\limits_{z \in \Zd \cap (\Ur \backslash U_{r, 2\ell(\ec)} + \cu)} \left(\Vert  \phim_{e_k}  \Vert_{\aL^2(z + \cu)} \right)$ & $\GO_s\left(C \ell(\ec)(\log r)^{\frac{1}{s}} \right)$  \\

 \hline
\end{tabular}
\end{adjustwidth}
\caption{A table of random variables $\X_1, \X_2, \Y_1$.}
\label{fig:factors}
\end{figure}

\section{Iteration estimate}
\label{sec:iteration}
In this part, we use \Cref{thm:TwoScale} to analyze the algorithm, and we give at first an $H^1, H^2$ à priori estimate.

\subsection{Proof of an $H^1, H^2$ estimate}
\begin{lemma}
\label{lem:H2Control}
In \cref{eq:iterative}, we have a control 
$$
\Vert \ub \Vert_{\aH^1(\Ur)} + \lambda^{-1}\Vert \ub \Vert_{\aH^2(\Ur)} \leq C(U,\Lambda,d) \Vert v - u \Vert_{\aH^1(\Ur)}.
$$
\end{lemma}
\begin{proof}
We test  the first equation $(\lambda^2 - \nabla \cdot \a \nabla)u_0 = - \nabla \cdot \a \nabla (u-v) $ in \cref{eq:iterative}  by $u_0$ and use the ellipticity condition to obtain 
\begin{eqnarray*}
\lambda^2 \Vert u_0 \Vert^2_{L^2(\Ur)} + \Lambda^{-1}\Vert \nabla u_0 \Vert^2_{L^2(\Ur)} &\leq & 
\lambda^2 \Vert u_0 \Vert^2_{L^2(\Ur)} + \int_{\Ur} \nabla u_0 \cdot \a \nabla u_0 \\
&=& \int_{\Ur} \nabla u_0 \cdot \a \nabla (u - v)  \\
&\leq & \Lambda  \Vert \nabla(v - u) \Vert_{L^2(\Ur)} \Vert \nabla u_0 \Vert_{L^2(\Ur)}\\
\Longrightarrow \Vert \nabla u_0 \Vert_{\aL^2(\Ur)} &\leq & \Lambda^2 \Vert \nabla(v - u) \Vert_{\aL^2(\Ur)}.
\end{eqnarray*}
We put back this term in the inequality, we also obtain that
\begin{equation}
\label{eq:L2U0}
\lambda\Vert u_0 \Vert_{\aL^2(\Ur)} \leq  \Lambda^{\frac{3}{2}} \Vert \nabla(v - u) \Vert_{\aL^2(\Ur)}.
\end{equation} 
Using this estimate, we obtain that of $\nabla \ub$ by testing $- \nabla \cdot \ab \nabla \ub = -\nabla \cdot \a \nabla (u-v-u_0)$ with $\ub$
\begin{eqnarray*}
\int_{\Ur} \nabla \ub \cdot \ab \nabla \ub  &=& \int_{\Ur} \nabla \ub \cdot \a \nabla (u-v-u_0) \\
\Longrightarrow \Vert \nabla \ub \Vert_{\aL^2(\Ur)} &\leq &  \Lambda^2 \Vert \nabla (u-v-u_0) \Vert_{\aL^2(\Ur)} \\
&\leq &   \Lambda^2 \Vert \nabla (u-v) \Vert_{\aL^2(\Ur)} +  \Lambda^2 \Vert \nabla u_0 \Vert_{\aL^2(\Ur)} \\
&\leq & C(U,\Lambda, d)\Vert \nabla (u-v) \Vert_{\aL^2(\Ur)}.  
\end{eqnarray*}

Finally, we calculate the $H^2$ norm of $\ub$. Because it is the solution of ${-\nabla \cdot \ab \nabla \ub = \lambda^2 u_0}$, we apply the classical $H^2$ estimate of elliptic equation (see  \cite[Theorem 6.3.2.4]{evans1998partial})
\begin{eqnarray*}
\Lambda^{-1} \Vert  \ub \Vert_{\aH^2(\Ur)} &\leq & \lambda^2 \Vert u_0 \Vert_{L^2(\Ur)} \\
(\text{ Using } \cref{eq:L2U0})&\leq & \lambda \Lambda^{\frac{3}{2}} \Vert \nabla(v - u) \Vert_{L^2(\Ur)} \\
\Longrightarrow \Vert \ub \Vert_{\aH^2(\Ur)} &\leq & \lambda \Lambda^{\frac{5}{2}} \Vert \nabla(v - u) \Vert_{\aL^2(\Ur)}.
\end{eqnarray*}
\end{proof}

\subsection{Proof of the main theorem}
With all these tools in hand, we can now prove \Cref{thm:main}. We denote by $\R(\ec, \mu, r,  \a, d, U, \vb)$ the right hand side of  \cref{eq:TwoScaleH1}, that is 
$$
\Vert v - w\Vert_{\aH^1(\Ur)} \leq \R(\ec, \mu, r,  \a, d, U, \vb).
$$

\begin{proof}
We take the first and second equations in the \cref{eq:iterative} and use the equation \cref{eq:main}
\begin{eqnarray*}
-\nabla \cdot \ab \nabla \ub &=& \lambda^2 u_0  \\
&=& f + \nabla \cdot \a \nabla (v + u_0) \\
&=& -\nabla \cdot \a \nabla (u - v - u_0).
\end{eqnarray*}
This is in the frame of \Cref{thm:TwoScale} thanks to the classical $H^2$ theory that $\ub \in H^2(\Ur)$. We apply \Cref{thm:TwoScale} with abuse of notation of the two scale expansion 
$$
w := \ub +  \sum_{k=1}^d \p{k} (\Ext(\ub) \star \zeta) \phim_{e_k},
$$
with $\Ext(\ub)$ satisfying \cref{eq:Extension}. Then we obtain that
\begin{equation}
\label{eq:IteraControl1}
\Vert w - (u - v - u_0)\Vert_{\aH^1(\Ur)} \leq \R(\ec, 0, r,  \a, d, U, \ub).
\end{equation}
The third equation of \cref{eq:iterative} $(\lambda^2 - \nabla \cdot \a \nabla) \ut = (\lambda^2 - \nabla \cdot \ab \nabla) \ub$ is also of the form of the \Cref{thm:TwoScale}, so we obtain
\begin{equation}
\label{eq:IteraControl2}
\Vert \ut - w \Vert_{\aH^1(\Ur)} \leq \R(\ec, \ec, r,  \a, d, U, \ub).
\end{equation} 
We combine this two estimates and use the triangle inequality to obtain
\begin{equation}
\label{eq:IteraControl3}
\Vert (v + u_0 + \ut) - u \Vert_{\aH^1(\Ur)} \leq \R(\ec, 0, r,  \a, d, U, \ub) + \R(\ec, \lambda, r,  \a, d, U, \ub).
\end{equation}  
It remains to see how to adapt $\R(\ec, \mu, r,  \a, d, U, \ub)$ in a proper way in the context of \cref{eq:main}. We plug in the formula in \Cref{lem:H2Control} to separate all the norms of $H^1$ and $H^2$ and use $0 < \mu < \lambda$.
\begin{align*}
\R(\ec, \mu, r, \a, d, U, \ub) &\leq  C(U,\Lambda,d) \left[\vphantom{\frac{1}{2}}  \lambda + \lambda \X_1 + \left(1 + \ell(\ec)^{\frac{1}{2}} \lambda^{\frac{1}{2}} \right) \X_2 \right. \\
 & \qquad \left. + \left(\ell(\ec)^{\frac{1}{2}} \lambda^{\frac{1}{2}} + 1 \right) \left(\lambda + \frac{1}{r}+ \frac{1}{\ell(\ec)}\right) \Y_1   \right] \Vert v - u \Vert_{\aH^1(\Ur)} .
\end{align*}
By checking \cref{fig:factors} and notice that the largest term is $\ell(\ec)^{-\frac{1}{2}} \lambda^{\frac{1}{2}} \Y_1$, so we obtain that the factor is of type $\GO_s\left(C(U,\Lambda,s, d)(\log r)^{\frac{1}{s}}  \ell(\lambda)^{\frac{1}{2}}\lambda^{\frac{1}{2}}) \right)$ as desired.
\end{proof}

\paragraph{Acknowledgment}
I am grateful to Jean-Christophe Mourrat for his suggestion to study this topic, helpful discussions and detailed reading of the article.


\bibliographystyle{abbrv}
\bibliography{UniformBoundRef}

\end{document}